\documentclass[10pt]{amsart}
\usepackage{texmaX,semmaX,semtkzX}
\usepackage{makecell}

\newcommand{\Langle}{\langle\!\langle}
\newcommand{\Rangle}{\rangle\!\rangle}

\newtheorem*{warning}{Warning}
\newtheorem*{convention}{Convention}

\usepackage{dsfont}
\def\mod{\quad\textup{mod }}
\def\JX{\Jac_X}

\def\L{L}

\def\ct{\cC_{\Theta}^{\textup{sf}}}

\def\mod{\quad\textup{mod }}
\advance\oddsidemargin -1cm
\advance\evensidemargin -1cm
\advance\textwidth 2cm
\advance\headsep -0.3cm
\advance\textheight 0.3cm 

\begin{document}

\title{Brauer relations, isogenies and parities of ranks}
\author{Alexandros Konstantinou}

\address{University College London, London WC1H 0AY, UK}
\email{alexandros.konstantinou.16@ucl.ac.uk}

\begin{abstract}
The present paper illustrates the utility of Brauer relations, Galois covers of curves and the theory of regulator constants in the context of studying isogenies between Jacobians and their relevance to the parity conjecture. This framework presents a unified approach, enabling the reconstruction of a diverse array of classical isogenies and the derivation of local expressions for Selmer rank parities, drawing from an extensive body of existing literature. These include the local expressions found in the works of Mazur--Rubin (dihedral extensions), Coates--Fukaya--Kato--Sujatha ($p^g$ isogenies), Kramer--Tunnell (quadratic twists of elliptic curves), Dokchitser--Maistret (Richelot isogenies), and Docking (prym construction).
\end{abstract}

\maketitle

\setcounter{tocdepth}{1}

\tableofcontents
\section{Introduction}\label{Sec_intro}

\subsection{Uniform results for isogenies and rank parity expressions}
The main motivation for this paper stems from \cite{tamroot} in which Brauer relations are used to establish the parity conjecture for a certain class of twists of elliptic curves. Building upon this groundwork, \cite{DGKM} introduces an expanded framework, enabling a broader scope by allowing $K$-automorphism subgroups on curves and pseudo Brauer relations. A Brauer relation is defined as a $\mathbb{Z}$-linear combination of subgroups of a finite group $G$, say $\sum_{i} H_{i} - \sum_{j} H_{j}'$, arising from a $G$-representation isomorphism
\[\bigoplus_{i} \textup{Ind}_{H_i}^{G}\mathbf{1} \cong \bigoplus_{j} \textup{Ind}_{H_j'}^{G}\mathbf{1} \ \ (H_i,H_j' \leq G). \] 

The notion of a pseudo Brauer relation provides an extension of these and is recalled in Definition \ref{def:pseud_brauer_curves} below. One of the main results in \cite{DGKM} asserts that, for a finite group $G$ acting on a curve $X$, a pseudo Brauer relation represented by $ \sum_{i} H_{i} - \sum_{j} H_{j}'$ induces an isogeny between Jacobians
\begin{equation} \label{isogeny_from_relation} \prod_{j} \textup{Jac}_{X/H_{j}'} \to \prod_{i} \textup{Jac}_{X/H_i}.  \end{equation}

Let us call such an isogeny \textit{pseudo Brauer verifiable}, see Definition \ref{brauer_verifiability_definition} for details. One of the central results in this paper establishes pseudo Brauer verifiability for a range of classical isogenies between Jacobians.
\begin{theorem}[Theorems \ref{pseudo_brauer_KT},  \ref{pseudo_brauer_richelot},  \ref{pseudo_brauer_verifiability_prym},  \ref{pseudo_brauer_verifiability_EC}, \ref{complementary_elliptic_curve}]  \label{pseudo-Brauer-verifiability}
In each of the following cases, there exists an isogeny $A \to A'$, and this isogeny is pseudo Brauer verifiable.
\begin{enumerate}
\item $A$ is the Weil restriction $\Res_{K(\sqrt{d})/K} \Jac_{Y}$ and $A'$ is $\Jac_Y \times \Jac_{Y^d},$
where $Y$ is a hyperelliptic curve and $Y^d$ its quadratic twist by $d\in K^\times$,
\item $A$ and $A'$ are Jacobians of C2D4 genus $2$ curves admitting a polarized $(2,2)$ isogeny (also known as a {Richelot isogeny}),
\item $A$ is $\Jac_{\tilde{Z}}$ and $A'$ is $\Jac_Z \times \textup{Prym}(\tilde{Z}/Z)$ where $\tilde{Z}$ is a (potentially trivial) twist of a curve admitting an unramified double cover to a trigonal curve $Z$,
\item $A$ and $A'$ are elliptic curves admitting an isogeny of prime degree $p$,
\item $A$ is $\textup{Jac}_{Y}$, where $Y$ is a curve of genus $2$ with a generic cover to an elliptic curve $Y\to E$ of prime degree $p$, $A'$ is $E \times E'$ for some complimentary elliptic curve $E'$. 
\end{enumerate} 
\end{theorem}
The isogenies listed in Theorem \ref{pseudo-Brauer-verifiability} are classically well-known. The key insight is that one can deduce their existence in a uniform way using a suitable pseudo Brauer relation.
\begin{remark} \label{remark:curves_aut_groups}
Writing $\phi:A \to A'$ to denote an isogeny in each of the cases listed in Theorem \ref{pseudo-Brauer-verifiability}, the following table records the pseudo Brauer relation $\Theta_{\phi}$ and the corresponding automorphism group $G_{\phi}$. For details regarding the curve $X_{\phi}$ being acted upon by $G_{\phi}$, see the relevant theorem.
$$\begin{array}{| c||c|c|c|}
\hline
\textup{Case} & G_{\phi} &\Theta_{\phi} & \textup{Relevant theorem}  \\
\hline
(1)& C_2 \times C_2 & \{e\} + 2(C_2\!\!\times\!\! C_2)  - C_2^a - C_2^b - C_2^c & \ref{pseudo_brauer_KT}\\
\hline
(2)&  D_{8} & C_2^a - (C_2^2)^a + C_2^b-(C_2^2)^b & \ref{pseudo_brauer_richelot}\\
\hline
(3)&  S_4 & C_2^2-D_8-S_3+S_4 & \ref{pseudo_brauer_verifiability_prym}\\
\hline
(4)&  C_p \rtimes C_m & C_m - (C_p \rtimes C_m) & \ref{pseudo_brauer_verifiability_EC}\\
\hline
(5)& S_{p} \times C_{2} & (S_{p-1}\times \{e\}) - (S_{p}\times \{e\}) - (S_{p-2}\times C_2) & \ref{complementary_elliptic_curve}\\
\hline
\end{array}$$
\smallskip
\end{remark}

\begin{remark} \label{brauer_verfiaibility}
We note that a Brauer relation for a group $G$ is automatically a pseudo Brauer relation for an automorphism group $G$ acting on a curve $X/K$. We note that the pseudo Brauer relations $\Theta_{\phi}$ in the cases listed (1)-(3) in Theorem \ref{pseudo-Brauer-verifiability} are Brauer relations for the underlying group $G_{\phi}$.\end{remark}

Adding on, the theory of {regulator constants} utilises isogenies induced from pseudo Brauer relations to provide local expressions for the parity of certain Selmer ranks. To be precise, in view of \cite[Corollary 7.3]{DGKM}, recalled in Theorem \ref{regulator_constant_rank_parity} below, to each pseudo Brauer relation $\Theta$ for $G$ and $X$, we can associate an expression of the form

\[ \sum_{\rho \in S_{\Theta,p}} m_{\rho} \equiv  (\textup{explicit local invariants}) \mod 2.\]
In this expression, $S_{\Theta,p}$ denotes a special set consisting of $\mathbb{Q}_p[G]$-representations (Definition \ref{def_special_subset}), while $m_{\rho}$ denotes the multiplicities of these in $\mathcal{X}_p(\textup{Jac}_{X})$, the dual $p^{\infty}$-Selmer group of $\textup{Jac}_{X}$. This specifies a mechanism for deriving local expressions. One of our main results is deduced from the application of this in various settings.

\begin{theorem}[Theorems \ref{local_expression_KT}, \ref{local_formula_Richelot}, \ref{prym_local_formula}, \ref{local_formula_EC}, \ref{Thm:local_formula_genus2_split}] \label{Thm:recoverable_formulae}
Let $A$ be one of the abelian varieties listed in Theorem \ref{pseudo-Brauer-verifiability}. Then, by applying Theorem \ref{regulator_constant_rank_parity} with $p=2$ in cases (1)-(3), and  with any rational prime $p$ in cases (4)-(5), we acquire
\[\textup{rk}_{p}(A) = \sum_{v \ \textup{place of} \ K} \textup{ord}_{p} \ \Lambda_{\Theta_{\phi}}(X_{\phi}/K_{v}) \mod 2, \ \]
where $\Theta_{\phi}$ is the pseudo Brauer relation given in Remark \ref{remark:curves_aut_groups}, while $\Lambda_{\Theta}(X_{\phi}/K_{v})$ is a {local} invariant as in Definition \ref{def:lemma:local_invariants}.
\end{theorem}

We note that Theorem \ref{Thm:recoverable_formulae} recovers local expressions found in the existing literature. In particular, the corresponding local expressions can be found in the works of Kramer \cite{MR597871} and Morgan \cite{morgan} for (1), V. Dokchitser--Maistret \cite{DM2019} for (2), Docking \cite{docking20212} for (3), Cassels (found in \cite[Appendix]{MR2167089}) for (4), Green--Maistret \cite{HollyCeline} for $p=2$ and Coates--Fukaya--Kato--Sujatha \cite{CFKS} for odd $p$ for (5). To our best knowledge, these are all known instances of local expressions in the case of Jacobians.

\subsection{Brauer relations and isogenies}
\label{subsec_isogenies_brauer}
The work of Kani and Rosen \cite{MR1000113} concerns the study of idempotent relations in the algebra of rational endomorphisms of arbitrary abelian varieties. Their central result in the setting of automorphisms on Jacobians proves that Brauer relations induce isogenies between Jacobians as in \eqref{isogeny_from_relation}. The usefulness of this theorem rests ultimately on one's ability to deduce the existence of an isogeny by presenting an explicit Brauer relation. As a result, applications of this work are widespread across the existing literature, see for example \cite{Bruin-Flynn}, \cite{MR1779608}, \cite{Paulhus} and \cite{Yu}. 

One noteworthy application includes Edixhoven's work  \cite{Edixhoven-verification} which presents an explicit Brauer relation for $\textup{GL}_{2}(\mathbb{F}_{p})$ in order to \textit{verify} the existence of an isogeny between Jacobians of certain modular curves. The existence of this isogeny was originally deduced in Chen's work by utilising the Selberg trace formula, see \cite[Theorem 1]{Chen-isogeny}. This provides the main stimulus for the contents of Theorem \ref{pseudo-Brauer-verifiability} mentioned above which asserts that a range of classical isogenies are pseudo Brauer verifiable in the following sense.

\begin{definition} \label{brauer_verifiability_definition} Let $Y_1,\dots,Y_m$ and $Y_1',\dots,Y_r'$ be curves defined over a number field $K$. In addition, suppose that there exists a $K$-rational isogeny $ \prod_{i=1}^{m} \textup{Jac}_{Y_{i}} \to \prod_{j=1}^{r} \textup{Jac}_{Y_{j}'}$. We say that a pseudo Brauer relation $\Theta$ for $G$ and a curve $X$ \textit{verifies an isogeny $\prod_{i=1}^{m} \textup{Jac}_{Y_{i}} \to \prod_{j=1}^{r} \textup{Jac}_{Y_{j}'}$} if we can find subgroups $H_1,\dots,H_m$ and $H_1',\dots,H_r'$ of $G$ such that
\begin{enumerate}
\item the quotient curves satisfy $X/H_i \cong Y_i$ and $X/H_j' \cong Y_j'$ for all $i$ and $j$,
\item and $\Theta = \sum_{i=1}^{m} H_{i} - \sum_{j=1}^{r} H_{j}'$ is the underlying pseudo Brauer relation.
\end{enumerate}
 In addition, we say that an isogeny $\prod_{i=1}^{m} \textup{Jac}_{Y_i} \to \prod_{j=1}^{r} \textup{Jac}_{Y_j'}$ is \textit{pseudo Brauer verifiable} if we can find a group $G$ acting on a curve $X$ and a pseudo Brauer relation $\Theta$ for these such that conditions (1)-(2) are satisfied.\end{definition}
  
\begin{remark} 
We note that the concept of pseudo Brauer verifiability is primarily focused on verifying the existence of \textit{some} isogeny between $\prod_{i} \textup{Jac}_{X/H_i}$ and $\prod_{j} \textup{Jac}_{X/H_j'}$. We deliberately refrain from making any specific references to an explicit morphism $\prod_{i} \textup{Jac}_{X/H_i} \to \prod_{j} \textup{Jac}_{X/H_j'}$ in this definition.
\end{remark}

It would be interesting to know whether \textit{any} isogeny between Jacobians is pseudo Brauer verifiable. A slightly more ambitious problem would be deriving similar results in the case of general abelian varieties. A result in this direction can be found in the work of Yu \cite[Corollary 10]{Yu}, where the author deduces an isogeny between $\textup{Res}_{F/K}A$ for $F/K$ an abelian extension and $\prod_{\chi \in \textup{Gal}(F/K)} A^{\chi}$, a product of its twists. This is done by using {twisted idempotent relations}, his generalisation of the Kani-Rosen relations. The underlying isogeny in this case is classically well-known and can be found in \cite{Kida}. 

We end this subsection with various examples of pseudo Brauer verifiable isogenies.

 \begin{example}[Modular Curves] Let $X(p)/\mathbb{Q}$ be the modular curve classifying elliptic curves with full level $p$ structure equipped with its usual action by $\textup{GL}_2(\mathbb{F}_{p})$. As already discussed, the main result presented in \cite{Edixhoven-verification} proves that the isogeny $\textup{Jac}_{X_{\textup{sp}}^{+}(p)} \to \textup{Jac}_{X_{0}(p)} \times  \textup{Jac}_{X_{\textup{nsp}}^{+}(p)}$ between certain modular Jacobians can be deduced from a Brauer relation for $\textup{PGL}_2(\mathbb{F}_{p})$. As a result, the underlying isogeny is pseudo Brauer verifiable.
 
 \end{example}
 
 \begin{example}[Fermat Curves, cf. \cite{MR1000113} \S 5] Let $X_m$ be the Fermat curve with affine model $x^m+y^m=1$. Let $K$ be any field containing $\zeta_p$ where $p|m$. Then, $C_p \!\times\! C_p$ acts on $X_m$ via $\sigma(x,y)=(\zeta_p x,y)$, while $\tau(x,y)=(x,\zeta_p y)$. Let $\{C_p^{i}\}_{i=0}^{p}$ denote all proper subgroups of $C_p\!\times\!C_p$, and write $X_i{'}$ to denote the (normalisation of) the curve with affine model  $ y_i^{m}=x_i^{im/p}(1-x_i^{m/p}). $ Then, the Brauer relation $\{e\} + p(C_p \!\times\! C_p) - \sum_{i=0}^{p} C_p^{i}$ for $C_p \!\times\! C_p$ and $X_m$ verifies an isogeny
 $ \textup{Jac}_{X_m} \times \textup{Jac}_{X_{m/p}}^{p} \hspace{-2mm} \to \prod_{i=0}^{p} \textup{Jac}_{X_i'}.$
 \end{example}
 
 \begin{example}[Hyperelliptic curve with extra involution] \label{extra_involution_example}
Let $X/K$ be a hyperelliptic curve which admits an affine model $y^2=f(x^2)$ for $f$ a square-free polynomial.
 The morphisms
 \begin{alignat*}{2} \phi_1: X &\to Y_1 =\{y_1^2=f(x_1)\}, \ \ \ \ \phi_2: X & &\to Y_2=\{y_2^2=x_2f(x_2)\}  \\
 (x,y) &\mapsto (x^2,y), \ \ \ \hspace{20mm} (x,y)& &\mapsto(x^2,xy)
\end{alignat*}
induce an isogeny $(\phi_1)^{*}+(\phi_2)^{*}:\textup{Jac}_{Y_1} \times \textup{Jac}_{Y_2} \to \textup{Jac}_{X}.$ By considering the $C_2 \times C_2$-action on $X/K$ given by separately negating $x$ and $y$, we acquire:
\begin{figure}[h!]
\centering\begin{tikzpicture}
    \node (Q1) at (0,0) {$K(x^2)$};
    \node (Q2) at (0,1.2) {$K(x)$};
    \node (Q3) at (-3, 1.2) {$K(Y_1) = K(x^2,y)$};
    \node (Q4) at (0, 2.4) {$K(X) = K(x,y)$};
	\node (Q5) at (3, 1.2) {$K(Y_2) = K(x^2, xy)$};
	    \draw (Q4)--(Q3);
	    \draw (Q4)--(Q5);
    \draw (Q4)--(Q2);
    \draw (Q2)--(Q1);
    \draw (Q3)--(Q1);
    \draw (Q5)--(Q1);
        \draw (Q5)--(Q1);
            \draw (Q5)--(Q1);
    \end{tikzpicture}\end{figure}
    \\
Since $\Xi=\{e\}+2(C_2\times C_2) - C_2^a-C_2^b-C_2^c$ is a Brauer relation for $C_2 \times C_2$, we deduce that the underlying isogeny is pseudo Brauer verifiable.
 \end{example}
 
 \begin{example}[Elliptic curve with $K$-rational $p$-torsion] \label{p-isogeny} Let $E$  be an elliptic curve with affine model $y^2=f(x)$ which has a $K$-rational $p$-torsion point $P$. Then, there exists an isogeny $\phi: E\to E' = \{v^2=g(u)\}$ with $\textup{Ker}(\phi) = \langle P \rangle$. By exploiting the $D_{2p}$-action on $E$ given by the translation-by-$P$ automorphism and the hyperelliptic involution, we acquire:
 \begin{figure}[h!]
\begin{center}\begin{tikzpicture}
    \node (Q1) at (0,0) {$K(u)$};
    \node (Q2) at (-1.5,1.3) {$K(E') = K(u,v)$};
    \node (Q3) at (+1.5,1.7) {$K(x)$};
    \node (Q4) at (0,3) {$K(E) = K(x,y)$};
    \draw (Q4)--(Q2);
    \draw (Q4)--(Q3);
    \draw (Q3)--(Q1);
    \draw (Q2)--(Q1);
    \end{tikzpicture} \end{center}\end{figure}
 \smallskip
 
In view of the Brauer relation $\Psi=\{e\}+2D_{2p}-2C_{2}-C_p$ for $D_{2p}$, we deduce that the $p$-isogeny $E \to E'$ is pseudo Brauer verifiable.
 
  \end{example}

\subsection{Local expressions, regulator constants and the parity conjecture} \label{subsect:known_local_formulae}

In addition, this paper addresses the systematic derivation of local rank parity expressions for Jacobians. The existence of such expressions is expected in view of the parity conjecture. This conjecture is based on expectations made by the Birch and Swinnerton-Dyer conjecture combined with the theoretical framework of $L$-functions.

\begin{conjecture}[Parity conjecture] For an abelian variety $A/K$ defined over a number field,
\[ \textup{rk}(A/K) \equiv \sum_{v \ \textup{place of} \ K} \mu(A/K_{v}) \mod 2,\] where $(-1)^{\mu(A/K_v)} = w(A/K_v)$, and $w(A/K_v)$ denotes the local root number of $A$ at a place $v$.
\end{conjecture}

Such global-to-local expressions have established themselves as highly valuable. Their usefulness stems from the relative ease with which local invariants can be computed, a stark contrast to the formidable challenge posed by global invariants such as the rank of an abelian variety. As a result, they prove to be a crucial step in many parity-related proofs, see for example \cite{CFKS}, \cite{tamroot}, \cite{DM2019} and \cite{KT}.

A notable instance of a global-to-local formula appears in the work of Mazur--Rubin \cite{MazurRubin} in the context of base changing elliptic curves to $D_{2p}$-extensions. The primary  impetus behind the development of Theorem \ref{Thm:recoverable_formulae} can be traced back to a result in \cite{MR2680426}. Within its pages, the authors employ their theory of regulator constants to give a local expression for the same rank parity as the formula appearing in Mazur--Rubin's work. This theory was subsequently extended in \cite{DGKM}, where the new framework allows for a broader scope by permitting the use of subgroups of $K$-automorphisms on curves and pseudo Brauer relations. We now recall the main result of this advancement.

\begin{notation} We write $\langle \cdot,\cdot \rangle$ for the standard inner product of complex characters of $G$, and we fix an arbitrary embedding $\mathbb{Q}_p \subseteq \mathbb{C}$ for the purpose of comparing characters. We write $\mathcal{X}_{p}(\textup{Jac}_{X})$ for the dual $p^{\infty}$-Selmer group of $\textup{Jac}_{X}$ and $\textup{rk}_{p}(\textup{Jac}_{X})$ for its $p^{\infty}$-Selmer rank. 
\end{notation}

\begin{theorem}[\cite{DGKM} Corollary 7.3] \label{regulator_constant_rank_parity}
Let $\Theta$ be a pseudo Brauer relation for $G$ and $X/K$, and in addition suppose that $\Omega^{1}(\textup{Jac}_{X})$ is self-dual as a $G$-representation. Then, for each rational prime $p$, 
\begin{equation*} \sum_{\tau \in S_{\Theta,p}} \langle \tau , \mathcal{X}_{p}(\textup{Jac}_{X}) \rangle  \equiv  \sum_{v \ \textup{place of} \ K} \textup{ord}_{p} \ \Lambda_{\Theta}(X/K_{v}) \mod 2,
\end{equation*} 
where $\Lambda_{\Theta}(X/K_{v})$ are {local} invariants as in Definition \ref{def:lemma:local_invariants} and $S_{\Theta,p}$ as in Definition \ref{def_special_subset}.
\end{theorem}

\begin{remark}
The self-duality assumption on $\Omega^{1}$ is automatically satisfied when all representations of $G$ are self-dual or $K$ has a real place (see \cite[Lemma 5.1]{DGKM}).
\end{remark}

We conclude this section by demonstrating applications of this theorem in examples considered in \S\ref{subsec_isogenies_brauer}.

\begin{example} 
Looking at Example \ref{extra_involution_example}, we write $C_2^a$ and $C_2^b$ to denote the subgroups generated by $(x,y) \mapsto (-x,y)$ and $(x,y) \mapsto (-x,-y)$ respectively. We let $\epsilon_a$ and $\epsilon_b$ be the non-identity, irreducible characters of $C_2\times C_2 $ on which $C_2^a$ and $C_2^b$ act trivially. Then, the special subsets $S_{\Theta,p}$ in Examples \ref{extra_involution_example} and \ref{p-isogeny} are as follows.
$$\begin{array}{|c||c|c|c|c|}
\hline
\textup{Example} & G&\textup{Brauer relation}&p&S_{\Theta,p}\\
\hline
 \ref{extra_involution_example} & C_2\times C_2& \Xi= \{e\}+ 2(C_2\!\!\times\!\! C_2 ) - C_2^a - C_2^b - C_2^c &2&\{\epsilon_a, \epsilon_b\}\\
\hline
\ref{p-isogeny} & D_{2p}& \Psi= 2C_2 + C_p - 2D_{2p} - \{e\}&p& \{\mathbf{1}\}\\
\hline
\end{array}$$
By applying Theorem \ref{regulator_constant_rank_parity} to Example \ref{extra_involution_example} with respect to the pseudo Brauer relation $\Theta$ and $p=2$, we acquire the local formula
\begin{equation*} \textup{rk}_{2}(\textup{Jac}_{X}) = \sum_{v \ \textup{place of} \ K} \ \textup{ord}_{2} \  \Lambda_{\Xi}(X/K_v) \mod 2. \end{equation*}
In addition, by applying Theorem \ref{regulator_constant_rank_parity} to Example \ref{p-isogeny} with respect to the pseudo Brauer relation $\Psi$ and any prime $p$, we acquire the local formula
\[\textup{rk}_{p}(E/K) = \sum_{v \ \textup{place of} \ K} \ \textup{ord}_{p} \  \Lambda_{\Psi}(E/K_v) \mod 2. \]
\end{example}

\subsection{Outline} This paper is organised as follows. In $\S2$, we provide relevant background results. This includes a review of the definitions of pseudo Brauer relations, regulator constants and some well-known facts concerning monodromy. The remaining sections share a consistent structure. In each of these, we delve into the study of the well--known isogeny specified by the section's title. Within each section, we establish pseudo Brauer verifiability for the isogeny in consideration, and then we utilise the underlying pseudo Brauer relation to derive a local expression using regulator constants.
\subsection{General notation and conventions} \label{General_notation}
Throughout this paper, we adhere to the following convention.

\noindent {\bf{\em{Convention.}}}
For each rational prime $p$, we fix an embedding $\Q_p \hookrightarrow{} \mathbb{C}$ for the purposes of comparing characters.

\medskip
\begin{tabular}{p{0.135\textwidth}p{0.81\textwidth}}
$X/K$ & a curve over a field $K$ (which is smooth, proper, not necessarily connected, see \S \ref{notation_for_curves})\\
$A/K$ & an abelian variety over a field $K$\\
$\textup{Res}_{F/K}A$ & the Weil restriction  of $A$ from a field extension $F$ to $K$\\
$\JX$ & the Jacobian variety of $X$, the identity component of the relative Picard functor $X/K$
\\
$\mathcal{X}_{p}(A)$ & $\text{Hom}_{\mathbb{Z}_{p}}(\varinjlim \text{Sel}_{p^n}(A), \mathbb{Q}_{p}/\mathbb{Z}_{p})\otimes \mathbb{Q}_{p}$, the dual $p^{\infty}$-Selmer group of $A$\\
$\rk_{p}A$ & the $p^{\infty}$-Selmer rank of $A$, equivalently $\textup{dim}(\mathcal{X}_{p}(A))$\\
$V_\ell(A)$& the $\mathbb{Q}_{\ell}$ vector space of the $\ell$-adic Tate module for $A$\\
$\Omega^{1}(A)$ & the $K$-vector space of regular differentials on $A$\\
$|\cdot|_{K}$, $|\cdot|_{v}$ & the unique extension of the the normalised absolute value on a local field $K$ (resp. $K_v$) to $\overline K$ (resp. $\overline{K_v}$)\\
$c(A)$ & the local Tamagawa number of $A$ defined over a non-Archimedean local field $\mathcal{K}$. When $\mathcal{K}$ is the completion of $K$ at a finite place $v$, then we write $c_{v}(A)$ to signify this.\\
$C(A, w)$ & $c(A) \cdot |\omega/\omega^{0}|_{\mathcal{K}}$ when $\mathcal{K}$ is non-archimedean and $\omega^{0}$ is a N\'eron exterior form on $A$; $\int_{A(\mathcal{K})} |\omega|$ when $\mathcal{K}=\mathbb{R}$; $2^{\text{dim}(A)}\int_{A(\mathcal{K})} |\omega \wedge \bar{\omega}|$ when $\mathcal{K}=\mathbb{C}$; in all cases $\omega$ is an exterior form for $A/\mathcal{K}$ \\
$\mu(X)$ & encodes whether $X$ is deficient, see \cite[Definition 5.5]{DGKM}\\
$X/H$ & quotient of the curve $X$ by a group $H\leq \Aut_K(X)$\\
$\langle\cdot,\cdot\rangle$ & the inner product on characters of $G$-representations \\
$\rho^*$, $\rho^H$ & the dual, resp. $H$-invariant ($H\leq G$) vectors, of a $G$-representation $\rho$\\

$\Theta $ & a pseudo Brauer relation relative to a $G$-representation $\mathcal{V}$, see Definition \ref{def:pseudo Brauer_relations} \\
$\mathcal{C}_{\Theta}(\mathcal{V})$ & regulator constant for $\mathcal{V}$, see Definition \ref{def:regulator_constants_for_pseudo_brauer}\\
$S_{\Theta,p}$ & set of self-dual and irreducible $\mathbb{Q}_{p}[G]$-subrepresentations of $\mathcal{X}_{p}$ with $\mathcal{C}_{\Theta}(\mathcal{X}_{p}) \equiv p \mod \mathbb{Q}_{p}^{\times 2}$, see Definition \ref{def_special_subset} \\

$\Lambda_{\Theta}(X/{\mathcal{K}})$ & arithmetic local invariant corresponding to pseudo Brauer relation $\Theta$, see
 Definition \ref{def:lemma:local_invariants} (cf. \cite[Definition 6.16]{DGKM})\\
$ \mathbb{P}^{1}_{x}/K$ & $\mathbb{P}^{1}$ together with a choice of $x$ which generates its function field over $K$ \\
 $D_{2p}$& the dihedral group of order $2p$ \\
 $C_m$ & cyclic group of order $m$
\end{tabular}\medskip

We will usually write $K$ to denote a number field.

\subsection{Notation and convention for curves} \label{notation_for_curves}
 Throughout this paper, a curve $X/K$ is defined to be a $K$-variety which is pure of dimension $1$ (equivalently, we suppose that all its irreducible components have dimension $1$).  We assume that all curves are smooth and proper but not necessarily connected (and neither are the connected components assumed to be geometrically connected). We refer the reader to \cite[Appendix A]{DGKM} for further details concerning these curves and their Jacobians.

We will call a curve $X/K$ \textit{nice} if it is connected, and in addition we say it is \textit{very nice} if its geometric components are defined over the base field $K$. Given a number field extension $L/K$ and a very nice curve $X/L$, we write ${}_{K}(X)$ to denote $X$ viewed as a scheme over $K$, i.e. the scheme $X$ equipped with the structure morphism $X \to \textup{Spec}(L) \to \textup{Spec}(K)$. 

\subsection*{Acknowledgements}  I would like to express my special thanks to Vladimir Dokchitser, not only for suggesting this project but also for his guidance and support which proved to be essential for the completion of this work. Thank you, Vladimir. I would also like to thank Adam Morgan for his work in laying the theoretical foundations for much of what will be presented here. Thank you, Adam.

\section{Background}
In this section, we review existing results on pseudo Brauer relations, their regulator constants, and the relevant local invariant $\Lambda_{\Theta}$ already established in \cite{DGKM}. To conclude, we present some well-known results on monodromy in the case of Galois covers of curves.
\subsection{Regulator constants, pseudo Brauer relations and isogenies}
 \label{subsect:regulator_constants}

\begin{notation} We write $L$ to denote any field of characteristic $0$ which admits an embedding into $\mathbb{C}$, $G$ is a finite group, and $\mathcal{V}$ is a self-dual $L[G]$-representation. 
\end{notation}

\begin{definition} \label{def:pseudo Brauer_relations}  Let $H_1,...,H_n$ and $H'_1,...,H_m'$ be (not necessarily distinct) subgroups of $G$.   We say that $\Theta = \sum_{i} H_{i} - \sum_{j} H_{j}'$ is a \textit{pseudo Brauer relation relative to $\mathcal{V}$} if there are $\mathbb{C}[G]$-representations $\rho_{1}$ and $\rho_{2}$, satisfying $\langle \rho_{1},  \mathcal{V}  \rangle  = \langle \rho_{2} , \mathcal{V} \rangle  =0$, such that
\begin{equation*}\label{eq:pseudo_brauer_equation_defining}
 \rho_1\oplus \bigoplus_{i} \textup{Ind}_{H_i}^{G} \mathbf{1} \cong \rho_2\oplus \bigoplus_{j} \textup{Ind}_{H_j'}^{G} \mathbf{1} . 
 \end{equation*}
\end{definition}
\begin{notation}\label{pairing_1_2_notat}
 Let $\Theta = \sum_{i}H_{i} - \sum_{j} H_{j}'$ be a pseudo Brauer relation relative to $\mathcal{V}$. Given a non-degenerate, $G$-invariant,  $\L$-bilinear pairing   $\Langle,\Rangle$  on  $\mathcal{V}$,   we denote by $\langle ,\rangle_{1}$ the  pairing
\[\langle ,\rangle_{1}=\bigoplus_i \frac{1}{|H_i|}\Langle,\Rangle\quad \textup{on the vector space}\quad\bigoplus_{i} \mathcal{V}^{H_{i}},\]
and define the pairing $\left \langle,\right \rangle_2$ on $\bigoplus_{j} \mathcal{V}^{H_{j}'}$ in an analogous way.  Given a basis $\mathcal{B}=\{v_i\}_i$  for $\bigoplus_{i} \mathcal{V}^{H_{i}}$, we  denote by $\langle \mathcal{B},\mathcal{B}\rangle_1$ the matrix  with $(i,j)^{th}$ entry $\langle v_{i},v_{j} \rangle_1 $, and define  $\langle \mathcal{B}',\mathcal{B}'\rangle_2$ for a basis $\mathcal{B}'$ of  $\bigoplus_{j} \mathcal{V}^{H_{j}'}$ similarly.
\end{notation}
\begin{definition} \label{def:regulator_constants_for_pseudo_brauer}
Let $\Theta=\sum_{i} H_{i} - \sum_{j} H_{j}'$  be a pseudo Brauer relation relative to $\mathcal{V}$, and let $\Langle,\Rangle$ be a non-degenerate, $G$-invariant,    $\L$-bilinear pairing on   $\mathcal{V}$ taking values in some field extension $\L'$ of $\L$.  Given bases $\mathcal{B}$ for $\bigoplus_{i} \mathcal{V}^{H_{i}}$ and $\mathcal{B}'$ for $\bigoplus_{j} \mathcal{V}^{H_{j}'}$, we define
\[ \mathcal{C}^{\mathcal{B},\mathcal{B}'}_{\Theta}(\mathcal{V}) =  \frac{ \text{det}\langle \mathcal{B},\mathcal{B} \rangle_{1}}{\text{det} \langle \mathcal{B}',\mathcal{B}' \rangle_{2}}\in \L'^\times. \] 
We then define the \textit{regulator constant of $\mathcal{V}$ relative to $\Theta$}, denoted $\mathcal{C}_{\Theta}(\mathcal{V}) $, to be the class of  $\mathcal{C}^{\mathcal{B},\mathcal{B}'}_{\Theta}(\mathcal{V}) $ in $\L'^{\times}/ \L^{\times 2}$ for any choice of bases $\mathcal{B}$, $\mathcal{B}'$.
\end{definition}

\begin{remark}As noted in \cite[Theorem 2.7]{DGKM}, the value of $\mathcal{C}_{\Theta}(\mathcal{V})$ is independent of the pairing $\Langle,\Rangle$ chosen on $\mathcal{V}$, is well-defined as an element in $L^{\times}/L^{\times 2}$ and is multiplicative in both $\Theta$ and $\mathcal{V}$.
\end{remark}

The present paper will only utilise pseudo Brauer relations relative to the $\ell$-adic Tate module $V_{\ell}(\textup{Jac}_{X})$ of the Jacobian variety of $X/K$, a curve being acted upon by $G$. 

\begin{definition}[\cite{DGKM} Definition 3.3] \label{def:pseud_brauer_curves} Let $G$ be a finite group which acts on a curve $X/K$ via $K$-automorphisms. We say that $\Theta$ is a \textit{pseudo Brauer relation for $G$ and $X$} if $\Theta$ is a pseudo Brauer relation relative to $V_{\ell}(\textup{Jac}_{X})$. 

\end{definition}

As already mentioned, any pseudo Brauer relation for $G$ and $X$ induces an isogeny between Jacobians. This isogeny is then utilised by Theorem \ref{regulator_constant_rank_parity} to extract certain rank parities. This can be accomplished by first finding a subset of self-dual $\mathbb{Q}_{p}[G]$-representations $S_{\Theta,p}$, and then the multiplicity of these representations in $\mathcal{X}_{p}(\textup{Jac}_{X})$.

\begin{definition}[cf. \cite{DGKM} Definition 2.9] \label{def_special_subset}
We write $\mathcal{R}_{G}$ to denote the set of all self-dual $\mathbb{Q}_{p}[G]$-representations which are either irreducible or of the form $\tau \oplus \tau^{*}$ for some irreducible $\tau$. Then, we define the subsets $S_{\Theta,p}(X) \subseteq \mathcal{R}_{G}$ to be
\begin{equation*}
S_{\Theta,p}(X) = \{ \rho \in \mathcal{R}_{G} \ | \ \langle \rho , \mathcal{X}_{p}(\textup{Jac}_{X}) \rangle >0 \ \textup{and} \ \textup{ord}_{p} \ \mathcal{C}_{\Theta}(\rho) \equiv 1 \mod 2\},
\end{equation*} 
where $\Theta$ is a pseudo Brauer relation for $G$ and $X$. Where no confusion lurks, we will simply write $S_{\Theta,p}$ to denote this set.
\end{definition}

We end this subsection by recording some instances where the regulator constant $\mathcal{C}_{\Theta}(\mathcal{V})$ is known to be trivial in which case $\mathcal{V}$ cannot lie in $S_{\Theta,p}$.

\begin{lemma}[cf. \cite{tamroot} Corollary 2.25, Lemma 2.26]  \label{trivial_regulator_constant} Let $\mathcal{V}$ be a self-dual $\L[G]$-representation, and let $\Theta=\sum_{i} H_{i} - \sum_{j} H_{j}'$ be a pseudo Brauer relation with respect to $\mathcal{V}$. In addition, suppose that either
 \begin{enumerate}
 \item $\langle \mathcal{V}, \L[G/H_{i}] \rangle = \langle \mathcal{V}, \L[G/H_{j}'] \rangle = 0$ for all $i,j$,
 \item $\mathcal{V} \otimes \bar{\L}$ is a symplectic $\bar{\L}[G]$-representation,
 \item $\mathcal{V} \otimes \bar{\L} \cong \tau \oplus \tau$ for some $\bar{\L}[G]$-representation $\tau$
 \item all $\bar{\L}$-irreducible constituents of $\mathcal{V} \otimes \bar{\L}$ are not-self dual
 \end{enumerate}
 Then, $\mathcal{C}_{\Theta}(\mathcal{V})=1$. 
 \end{lemma}

\subsection{The local invariant $\Lambda_\Theta(X/\mathcal{K})$}
We recall the definition of the {local arithmetic invariant} $\Lambda_{\Theta}(X/\mathcal{K})$ attached to a curve $X$ which is defined over a local field $\mathcal{K}$. This appears on the right hand side of Theorem \ref{regulator_constant_rank_parity}.

\begin{notation}
We continue to write $\Theta=\sum_{i} H_{i} - \sum_{j} H_{j}'$ to denote a pseudo Brauer relation for $G$ and $X$. In addition, we fix bases $\mathcal{B}_{1}$ and $\mathcal{B}_{2}$ for $\Omega^{1}(\prod_{i}  \Jac_{X/H_i})$ and $\Omega^{1}(\prod_{j} \Jac_{X/H_j'})$, respectively and write $\omega(\mathcal{B}_{1})$ to denote the exterior form on $\prod_i \Jac_{X/H_i}$ given by the wedge product of elements in $\mathcal{B}_{1}$, and similarly for $\omega(\mathcal{B}_{2})$.
\end{notation}
 
\begin{definition} \label{def:lemma:local_invariants} 
Suppose that $\Omega^{1}(\JX)$ is a $\mathcal{K}[G]$-representation which is realisable over $\mathbb{Q}$. For any choice of bases $\mathcal B_1, \mathcal B_2$ for $\Omega^{1}(\prod_{i} \Jac_{X/H_i})$ and $\Omega^{1}(\prod_{j} \Jac_{X/H_j'})$, we define
\[ \Lambda_{\Theta}(X/{\mathcal{K}})= \frac{C( \prod_{i} \Jac_{X/H_{i}},\omega(\mathcal{B}_{1})) \ }{C( \prod_{j} \Jac_{X/H_{j}'},\omega(\mathcal{B}_{2})) } \cdot \frac{\prod_{i}\mu({X/H_{i}})}{\prod_{j}\mu( {X/H_{j}'})}\cdot  \
\Biggl{|}\sqrt{\frac{\ct(\Omega^{1}(\JX))}{\mathcal{C}_{\Theta}^{\mathcal{B}_1,\mathcal{B}_{2}}(\Omega^{1}(\JX))}}\Biggr{|}_{\mathcal{K}},
\] where $C$ and $\mu$ are as in \S\ref{General_notation},  and ${\ct(\Omega^{1}(\JX))}$ is the unique, square-free integer satisfying ${\ct(\Omega^{1}(\JX))} \equiv \mathcal{C}_{\Theta}(\Omega^{1}(\JX)) \mod \mathbb{Q}^{\times 2}$.
\end{definition}

\begin{remark} For the purposes of the present paper, it suffices to consider the definition of $\Lambda_{\Theta}(X/\mathcal{K})$ specialised to the case where $\Omega^{1}$ is realisable over $\mathbb{Q}$. For a slightly more general definition of $\Lambda_{\Theta}(X/\mathcal{K})$ allowing $\Omega^{1}$ to be self-dual, see \cite[\S 5]{DGKM}.

\end{remark}

\subsection{Galois closures, monodromy and Riemann-Hurwitz} \label{subsec_monodromy} We now record some well-known results concerning monodromy. These tools will be of great importance when constructing Galois covers using a {down-to-top} approach by taking the Galois closure of a branched cover of curves.

\begin{warning}
For the remainder of this subsection \textbf{only}, our primary assumption will be that all curves are geometrically connected and, unless otherwise stated, defined over $\mathbb{C}$. We note that this deviates from the ongoing conventions in \S \ref{notation_for_curves}.
\end{warning}

 We adopt a similar notation to the one used in \cite{decomposing_jacobians} to record the ramification in the fibre of a branch point and the cycle type of a permutation.
\begin{definition} \label{Def:ramification_structure}  Let $f: Z \to \mathbb{P}^{1}$ be a morphism of complex curves, and let $\mathcal{B}_{f}=\{b_{i}\}_{i=1}^{r} \subset \mathbb{P}^{1}(\mathbb{C}) $ be its branch locus. Suppose that the fibre $f^{-1}(b_i)$ contains $n_j$ points each with ramification index $e_j>1$, and suppose that $j$ takes values in $\{1,\dots,k\}$. Then, we write $R_{b_i}=\prod_{j=1}^{k} (n_j)^{e_j} $ to denote the \textit{ramification structure of $f$ at the point $b_i$}. The \textit{ramification structure of $f$} is the multi-set given by $\{R_{b_i}\}_{i=1}^{r}$.
\end{definition}

\begin{definition} Let $\sigma$ be a permutation in $S_{n}$ expressed as a product of disjoint cycles. Suppose that $\sigma$ contains $m_j$ cycles each of length $t_j >1$ where $j$ takes values in $\{1,\dots,k\}$. Then, we write $r_{\sigma} = \prod_{j=1}^{k} (t_j)^{m_j}$ to denote the \textit{cycle type of $\sigma$}.
\end{definition}

\begin{example}
Let $\sigma \in S_7$ be the element  $(1,2)(3,4)(5,6,7)$. Then, $r_{\sigma}=(2)^2(3)^1$. 
\end{example}

\begin{theorem}[\cite{Miranda}] \label{theorem_miranda}Let $f:Z \to \mathbb{P}^{1}$ be a degree $n$ cover of complex curves branched at $r$ points $\{b_i\}_{i=1}^{r}$, and write $F:X_Z \to \mathbb{P}^1$ to denote its Galois closure. Then, there exists an $r$-tuple $(\sigma_1,\dots,\sigma_r)$ of elements in $S_n$ such that the ramification structure at $b_i$ matches the cycle type of $\sigma_i$. In addition, the Galois group of $X_Z \to \mathbb{P}^1$ is a transitive subgroup of $S_n$ generated by the $\sigma_i$.

\end{theorem}

The tuple $(\sigma_1,\dots,\sigma_r)$ afforded by Theorem \ref{theorem_miranda} is usually called the \textit{monodromy tuple of $F$}.

\begin{remark} \label{remark:decomposition_groups} As detailed in \cite[Proposition 3.18 \& Remark 3.19]{decomposing_jacobians}, the cyclic group $\langle \sigma_i \rangle$ afforded by Theorem \ref{theorem_miranda} coincides (up to conjugacy) with the decomposition group of any point in the fibre $F^{-1}(b_i)$. 
\end{remark}
The following theorem gives a description for the $G$-module structure of the $\ell$-adic Tate module of the Jacobian variety of a complex curve $X$ in terms of the of monodromy of the branched cover $X \to Y$. Given an element $\sigma \in G$, we write $\mathbf{1}_{\sigma}^{*}$ to denote the permutation representation $\textup{Ind}_{\langle \sigma \rangle}^{G} \mathbf{1}$.

\begin{theorem}[\cite{Ellenberg} Proposition 1.1] \label{equivariant_riemann_hurwitz}
Let $F:X \to Y$ be a Galois cover of complex curves which is branched at $\{b_{i}\}_{i=1}^{r}$, and write $G$ to denote its Galois group. In addition, suppose that $g_i$ is a generator for the decomposition group of any point in $F^{-1}(b_i)$. Then, for any prime $\ell$, the following $G$-representations 
\[V_{\ell}(\textup{Jac}_{X}) \ \ \textup{and} \ \ \mathbf{1}^{\oplus 2} \oplus (\mathbf{1}_{e}^{*})^{\oplus(2g(Y)-2)} \oplus \bigoplus_{i=1}^{r}(\mathbf{1}_{e}^{*} \ominus \mathbf{1}_{g_i}^{*}), \]
become isomorphic after extending scalars to $\mathbb{C}$, where $g(Y)$ denotes the genus of the curve $Y$.
\end{theorem}

For the remainder of this section, we write $K$ to denote a number field or a local field of characteristic $0$.
\begin{corollary} \label{equivariant_riemann_hurwitz_gc}
Let $X/K$ be a geometrically connected curve. Let $G$ be a subgroup of $K$-automorphisms of $X$. Then, for any prime $\ell$, Theorem \ref{equivariant_riemann_hurwitz} holds for $X/K$.
\end{corollary}
\begin{proof}
Base-changing along a field embedding $K \hookrightarrow{} \mathbb{C}$, it suffices to prove the result for $K$ replaced with $\mathbb{C}$ and $X$ a complex curve. This is precisely Theorem \ref{equivariant_riemann_hurwitz}.
\end{proof}
\begin{corollary} \label{eq_riemann_hurwitz_gc} Let $f:Z \to \mathbb{P}^{1}$ be a $K$-rational morphism of geometrically connected curves. We write $\mathcal{B}_{f} = \{b_i\}_{i=1}^{r}$ to denote its branch locus. Let $F:X_Z \to \mathbb{P}^{1}$ be the Galois closure of $f$, and in addition suppose that $X_Z/K$ is geometrically connected. Then, for any prime $\ell$, the $G$-representations
\[V_{\ell}(\textup{Jac}_{X_{Z}}) \ \ \textup{and} \ \ \mathbf{1}^{\oplus 2} \ominus (\mathbf{1}_{e}^{*})^{\oplus 2} \oplus \bigoplus_{i=1}^{r}(\mathbf{1}_{e}^{*} \ominus \mathbf{1}_{\sigma_i}^{*}), \]
become isomorphic after extending scalars to $\mathbb{C}$, where the $\sigma_i$ coincide with the permutations afforded by Theorem \ref{theorem_miranda}. 

\end{corollary}

\begin{proof}
By \cite[Corollary 3.9.3(b)]{algebraic_function_fields}, the branch locus of $Z \to \mathbb{P}^{1}$ coincides with the branch locus of its Galois closure $X_Z \to \mathbb{P}^{1}$, while by Remark \ref{remark:decomposition_groups} we deduce that the groups generated by the $\sigma_i$ are (up to conjugacy) equal to the decomposition at any point in the fibre of $X_Z \to \mathbb{P}^{1}$ at $b_i$. As a result, the claim follows by Corollary \ref{equivariant_riemann_hurwitz_gc}. 

\end{proof}
\section{Hyperelliptic Jacobians over $K(\sqrt{d})$}\label{Sec_KT}
In the following section, we provide a local expression for the $2^{\infty}$-Selmer rank of the Jacobian of a hyperelliptic curve $Y/K$ base changed to a quadratic field extension $L=K(\sqrt{d})$. We first prove pseudo Brauer verifiability for the isogeny 
\[ \textup{Res}_{L/K} \textup{Jac}_{Y/L} \to \textup{Jac}_{Y} \times \textup{Jac}_{Y^{d}}, \] where $Y^{d}$ denotes the quadratic twist of $Y$ by $d$. We write $Y_{L}$ to denote $Y \times_{K} L$ the base change to $L$, and as in \S\ref{notation_for_curves}, we write $X={}_{K}(Y_{L})$ to denote the curve obtained from $Y_{L}$ by forgetting the $L$-structure.

\begin{theorem} \label{Proposition:KT} Let $Y/K$ be a hyperelliptic curve with affine model $y^2=f(x)$, where $f$ is a square-free polynomial. Then, the function field inclusions $K(x) \subseteq K(x,y) \subseteq L(x,y)$ fit into a Galois diagram

\begin{center}
\begin{tikzcd}
                                & X \arrow[ld, no head] \arrow[d, no head] \arrow[rd, no head] &                             \\
Y=X/C_2^a \arrow[rd, no head] & Y^{d}=X/C_2^b \arrow[d, no head]                             & X/C_2^c \arrow[ld, no head] \\
                                & \mathbb{P}_{x}^{1}=X/(C_2\times C_2)                       &                            
\end{tikzcd}
\end{center}
where $C_{2}^{a},C_{2}^{b},C_{2}^{c}$ denote all proper subgroups of $C_{2} \times C_{2}$.
\end{theorem}

\begin{proof}
The function field $K(X)$ coincides with the biquadratic extension of $K(x)$ obtained by adjoining $y$ and $\sqrt{d}$. By \cite[Remark A.31]{DGKM}, this induces a $C_2 \times C_2$-action on $X$ via $K$-automorphisms obtained by separately negating $y$ and $\sqrt{d}$. The quotient curves agree with what is asserted by the claim in view of \cite[Remark A.30]{DGKM}. Indeed, two of the intermediate field extensions are $K(x,y)/(y^2-f(x)), K(x,\sqrt{d}y)/(dy^2-f(x))$ which coincide with the function fields of $Y$ and $Y^{d}$ respectively. \qedhere
\end{proof}

\begin{theorem}\label{pseudo_brauer_KT}
The pseudo Brauer relation $\Theta = \{e\}+2(C_2\times C_2) - C_2^a - C_2^b-C_2^c$ for $C_2 \times C_2$ and $X$ verifies an isogeny $\textup{Res}_{L/K} \textup{Jac}_{Y/L} \to \textup{Jac}_{Y/K} \times \textup{Jac}_{Y^{d}/K}$.

\end{theorem}

\begin{proof}
The function fields obtained by taking $C_2^c$ and $(C_2 \times C_2)$-invariance in Theorem \ref{Proposition:KT} are $K(x,\sqrt{d})$ and $K(x)$ respectively. Therefore, the quotients of $X$ by $C_2^c$ and $C_2 \times C_2$ are both curves of genus $0$. Therefore, the Brauer relation  $\Theta = \{e\}+2(C_2\times C_2) - C_2^a - C_2^b-C_2^c$ for $C_2 \times C_2$ induces an isogeny $\textup{Jac}_{X/K} \to \textup{Jac}_{Y/K} \times \textup{Jac}_{Y^d/K}$. By \cite[Lemma A.22]{DGKM}, it follows $\textup{Jac}_{X/K} = \textup{Res}_{L/K} \textup{Jac}_{Y/L}$ from which the result follows. \qedhere 
\end{proof}

\begin{theorem} \label{local_expression_KT} Let $Y/K$ be a hyperelliptic curve and  $Y^{d}/K$ the twist of $Y$ by $L=K(\sqrt{d})$. Then,

\[\textup{rk}_{2} \ \textup{Jac}_{Y}/K + \textup{rk}_{2} \ \textup{Jac}_{Y^d}/K \equiv \sum_{v \ \textup{place of} \ K} \textup{ord}_{2} \ \Lambda_{\Theta}(X/K_{v}) \mod 2, \]
where ${\Theta}$ is the pseudo Brauer relation for $C_2\times C_2$ and $X$ afforded by Theorem \ref{pseudo_brauer_KT}.
\end{theorem}
\begin{proof}

 We write $\mathcal{X}_{2}$ to denote $\mathcal{X}_{2}(\textup{Jac}_{X})$. We write $\epsilon_a$ and $\epsilon_b$ to denote non-identity characters of $C_2 \times C_2$ on which $C_2^a$ and $C_2^b$ act trivially. It follows that $\mathcal{X}_{2}\cong \epsilon_a^{\textup{rk}_{2}\textup{Jac}_{Y}} \oplus \epsilon_b^{\textup{rk}_{2}\textup{Jac}_{Y^{d}}}$. Therefore, $S_{\Theta,2}\subseteq \{ \epsilon_a, \epsilon_b\}$ (see Definition \ref{def_special_subset}). An elementary calculation reveals that $\mathcal{C}_{\Theta}(\epsilon_a) \equiv \mathcal{C}_{\Theta}(\epsilon_b) \equiv 2 \mod \mathbb{Q}^{\times 2}$ from which we deduce $S_{\Theta,2}=\{  \epsilon_a, \epsilon_b \}.$ 
The local expression asserted by the claim follows upon applying Theorem \ref{regulator_constant_rank_parity} to $\Theta$ afforded by Theorem \ref{pseudo_brauer_KT} with $p=2$.  \qedhere

\end{proof}
\section{Richelot Isogeny}\label{Sec_Richelot}
In this section, we use Brauer relations and regulator constants to deduce a local expression for the rank parity of the Jacobian of a genus $2$ curve $\textup{Jac}_{X}$ admitting a Richelot isogeny under some mild assumption on the Galois group of its $2$-torsion field (see Definition \ref{def:C2D4_curve} below). We first collect some useful yet well-known facts concerning Richelot isogenies, see for example \cite{Bruin} and \cite{correspondence} for details.

\begin{lemma} \label{existence_richelot} Let $Z/K$ be genus $2$ curve with affine Weierstrass model $y^2 = F(x)$, where $F(x)$ is a square-free polynomial of degree $6$ admitting a factorisation \[ F(x)= F_1(x)F_2(x)F_3(x) \] into 3 coprime quadratics subject to the condition that the set $\{F_1,F_2,F_3\} $ is stable under the action of $\textup{Gal}(\bar{K}/K)$. Then, there exists an isogeny $ \phi: \textup{Jac}_{Z} \to A$, where $A$ is a principally polarised abelian variety such that $\phi^{\vee} \circ \phi = [2]_{\textup{Jac}_{Z}}.$
\end{lemma}
\begin{proof}
See \cite[Proposition 8.2.3]{correspondence}
\end{proof}

We say that $\{F_1,F_2,F_3\}$ is a \textit{quadratic splitting} for $Z$. In view of \cite[Lemma 4.1]{Bruin}, a Jacobian of a genus $2$ curve admits a $(2,2)$-isogeny only when $Z$ admits a quadratic splitting. For the remainder of this section, we will assume that $\textup{Gal}(K(\textup{Jax}_{Z}[2])/K)$ is a $2$-group. Since $\textup{Gal}(K(\textup{Jax}_{Z}[2])/K)=\textup{Gal}(F)$, this assumption is equivalent to $\textup{Gal}(F) \leq D_8 \times C_2$, where $D_8$ denotes the dihedral group of order $8$. Following the terminology introduced in \cite{DM2019}, we refer to these curves as \textit{C2D4 curves} (albeit acknowledging the discrepancy in notation regarding dihedral groups.).

\begin{definition}\label{def:C2D4_curve} Let $Z/K$ be a genus $2$ curve admitting a quadratic splitting $\{F_1,F_2,F_3\}$. We say that $Z/K$ is a \textit{C2D4 curve} if $\textup{Gal}(F_1F_2F_3) \subseteq D_{8} \times C_{2}$ as a permutation group on $6$ roots, in which case the two factors $D_{8}$ and $C_{2}$ act separately on $4$ and $2$ roots respectively. 
\end{definition} 

\begin{convention} We note that the individual quadratic factors $F_i$ need not be defined over $K$. In particular, under the assumption that $Z$ is a $C2D4$ curve, then two of the quadratic factors could be Galois conjugate over a quadratic field extension $L=K(\sqrt{d})$. When this happens, we adhere to the convention that $F_1$ and $F_2$ are Galois conjugate over $K(\sqrt{d})$, while $F_3$ is necessarily defined over $K$. 
\end{convention}

\begin{notation} \label{choice_of_d} For the remainder of this section, for any $C2D4$ curve $Z$ we fix the following constant $d=d(Z)\in K^{\times}$. If $Z$ admits a quadratic splitting in which $2$ quadratic factors are Galois conjugate over a quadratic extension $L/K$, then we fix $d\in K^{\times}$ to be any generator for $L=K(\sqrt{d})$. If all quadratic factors are defined over $K$, then we fix $d=1$. 
\end{notation}

If we assume that the quadratic splitting is non-singular in the sense of 
\cite[Definition 8.2.4]{correspondence}, then \cite[Proposition 4.3]{Bruin} asserts that the target of this isogeny is the Jacobian of a $C2D4$ curve with an affine model $\tilde{Z}_{d}=\{dy^2=G_1(x)G_2(x)G_3(x) \}$. The polynomials $G_i$ are defined as follows.
\begin{definition} \label{target_of_isogeny} Let $\{F_1,F_2,F_3\}$ be a non-singular quadratic splitting for a genus $2$ curve, and let $F_i(x) = \sum_{j=0}^{2} f_{i,j}x^{j}$. Then, for $(i,j,k)=(1,2,3), (2,3,1), (3,1,2)$, we define 
\[G_{i}({x}) = 
\det\begin{pmatrix}
f_{1,0} & f_{1,1} & f_{1,2}\\
f_{2,0} & f_{2,1} & f_{2,2}\\
f_{3,0} & f_{3,1} & f_{3,2}
\end{pmatrix}^{-1}  \cdot \bigg{(} F_k(x)\frac{\mathrm{d}}{\mathrm{d}x}F_j(x) - F_j(x)\frac{\mathrm{d}}{\mathrm{d}x}F_k(x) \bigg{)}.  \] 
\end{definition}

A classical way of deducing the existence of an isogeny $\textup{Jac}_{Z} \to \textup{Jac}_{\tilde{Z}_{d}}$ is via a $(2,2)$-correspondence $\Gamma_{d}$ between $Z$ and $\tilde{Z}_{d}$.  This is defined as follows.

\begin{definition}\label{correspondence_definition} Let $Z/K$ be a genus $2$ curve admitting a non-singular quadratic splitting $\{F_1,F_2,F_3\}$. By expressing $\tilde{Z}_{d}$ in terms of $(\tilde{x},\tilde{y})$ instead of $(x,y)$, the \textit{Richelot correspondence} on $Z$ and $\tilde{Z}_{d}$ is the curve $\subseteq Z \times \tilde{Z}_{d}$ given by the equations
\begin{equation*}
  \Gamma_{d}:\begin{cases}
    \hfill F_{1}(x)G_{1}(\tilde{x}) + F_{2}(x)G_{2}(\tilde{x}) &= 0\\
   \hfill  F_{1}(x)G_{1}(\tilde{x})(x-\tilde{x}) &=\sqrt{d}\tilde{y}y\\
      \hfill  F_{2}(x)G_{2}(\tilde{x})(x-\tilde{x}) &=-\sqrt{d}\tilde{y}y\\
  \end{cases} 
  \end{equation*} 
  
\end{definition}
\begin{remark} \label{involutions}
The projection morphisms
\begin{alignat*}{2}
\pi_1: \Gamma_d &\to Z  \hspace{1.5cm} \textup{and} \hspace{1.5cm} \pi_2: \Gamma_d &&\to \tilde{Z}_{d} \\
(x,y,\tilde{x},\tilde{y}) &\mapsto (x,y) \hspace{3cm} (x,y,\tilde{x},\tilde{y}) &&\mapsto(\tilde{x},\tilde{y})
\end{alignat*}
induce an isogeny $\textup{Jac}_{Z} \to \textup{Jac}_{\tilde{Z}_{d}}$ via the composition $(\pi_2)_{*}\pi_1^{*}$. It is important to observe that the specific choice of $d$ as indicated in Notation \ref{choice_of_d} ensures that the induced isogeny is defined over $K$. This is because both $\Gamma_d$ and the $\pi_i$ are defined over $K$ when this choice is made. This need not be true for an arbitrary choice of $d$.
\end{remark}
In the following Proposition, we utilise the Richelot correspondence to prove pseudo Brauer verifiability for a Richelot isogeny between $C2D4$ curves. As always, we write $\mathbb{P}^{1}_{x}$ to denote $\mathbb{P}^{1}$ together with a choice of $x$ which generates its function field over $K$. 
\begin{proposition} \label{automorphisms_on_richelot} Let $Z/K$ be a $C2D4$ curve, and suppose that $\{F_1,F_2,F_3\}$ is a non-singular quadratic splitting for it. Suppose (as we may) that $F_1$ and $F_2$ are {monic} polynomials given by $F_i(x)=x^2+f_{i,1}x+f_{i,2}$. Then, define
\begin{equation*}{T}(x) = \begin{cases} \frac{-f_{1,2}f_{2,1}+f_{1,1}f_{2,2}+(f_{1,1}-f_{2,1})x^2}{(f_{1,2}-f_{2,2})+(f_{1,1}-f_{2,1})x},& \text{if } f_{1,1} \neq f_{2,1},\\
x(f_{1,1}+x),& \text{if } f_{1,1} = f_{2,1}.
\end{cases}
\end{equation*}
Then, it follows that
\begin{enumerate}
\item $T(x)$ is a $K$-rational function,
\item The function field extension $K(T) \subseteq K(x) \subseteq K(Z) \subseteq K(\Gamma_d)$ is Galois with Galois group $D_8 = \langle s, r \ | \ r^4=s^2=1, rs=r^3s \rangle$,
\item Assuming (as we may) that $Z=X/\langle s \rangle$, then these function field inclusions fit into a Galois diagram
\end{enumerate}
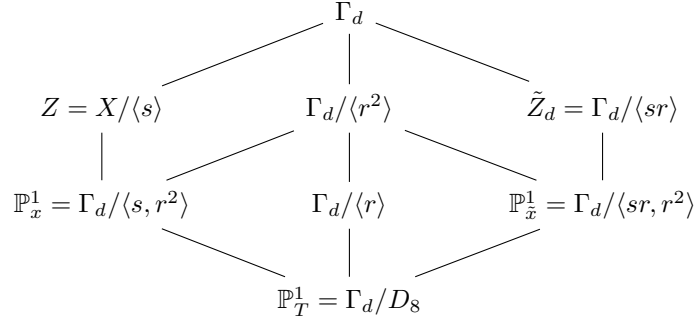
\begin{center}
\begin{tikzcd}
                                       & \Gamma_d \arrow[ld, no head] \arrow[d, no head] \arrow[rd, no head]                &                                                \\
Z =X/\langle s \rangle \arrow[d, no head]                   & \Gamma_d/\langle r^2 \rangle \arrow[ld, no head] \arrow[d, no head] \arrow[rd, no head] & \tilde{Z}_{d}=\Gamma_d/\langle sr \rangle \arrow[d, no head]        \\
\mathbb{P}^{1}_{x}=\Gamma_d/\langle s , r^2 \rangle \arrow[rd, no head] & \Gamma_d/\langle r \rangle \arrow[d, no head]                                           & \mathbb{P}^{1}_{\tilde{x}}={\Gamma_d/\langle sr,r^2 \rangle} \arrow[ld, no head] \\
                                       & \mathbb{P}^{1}_{{T}}=\Gamma_d/D_{8}                                                              &                                               
\end{tikzcd}
\end{center}
\end{proposition}

\begin{proof}
If the individual $F_i$ are defined over $K$, then (1) is clear. If $F_1$ and $F_2$ are quadratic conjugate over $K(\sqrt{d})$, then the non-trivial element $\sigma \in \textup{Gal}(K(\sqrt{d})/K)$ acts on the coefficients by permuting $f_{1,j} \leftrightarrow f_{2,j}$. Since $T(x)$ is $\sigma$-stable, then $T(x)$ is $K$-rational. This completes the proof of $(1)$.

In order to prove $(2)$, we first fix explicit generators for the function field of $\Gamma_d$. We fix $\tilde{y} \in \overline{K(\tilde{x})}$ by letting $\tilde{y}^2=G_1(\tilde{x})G_2(\tilde{x})G_3(\tilde{x})$. Then, by the description of $\Gamma_d$ given in Definition \ref{correspondence_definition}, we deduce that $K(\Gamma_d) = K(x,y,\tilde{x},\sqrt{d}\tilde{y})$. 
It follows that all $3$ function fields $K(x,y)=K(Z), K(x,\tilde{x})$ and $ \ K(\tilde{x},\sqrt{d}\tilde{y})=K(\tilde{Z}_{d})$  admit degree $2$ embeddings into $K(\Gamma_d)$.  These correspond to three involutions on $\Gamma_d$, say $g_1,g_2$ and $g_3$ respectively, and let $H = \langle g_1,g_2,g_3 \rangle$. Since $g_1$ and $g_2$ fix $x$, then they fix $T(x)$. Adding on, one checks that
\begin{equation*}
T(x) = \begin{cases} \frac{f_{2,1} {F_1(x)} - f_{1,1}{F_2(x)}}{{F_2(x)}-{F_1(x)}},& \text{if } f_{1,1} \neq f_{2,1}\\
\frac{f_{2,2} {F_1(x)} - f_{1,2}{F_2(x)}}{{F_2(x)}-{F_1(x)}},& \text{if } f_{1,1} = f_{2,1}.
\end{cases} 
\end{equation*}
Since $\frac{F_1(x)}{F_2(x)}=\frac{-G_2(\tilde{x})}{G_1(\tilde{x})}$, we deduce that $T$ is also a function in $\tilde{x}$, and since $g_3$ fixes $\tilde{x}$, then it must fix $T$. This shows that $ K(T)  \subseteq K(\Gamma_{d})^{H}$, and since $[K(\Gamma_{d}):K(T)]=8$, we deduce that $|H| \leq 8$. Since $K(x)$ isn't fixed by $H$ (since $g_3$ doesn't fix $x$), we deduce that $|H|> 4$. Therefore $|H|=8$ and in addition $K(\Gamma_{d})^{H} = K(T)$.

The claim in (2) will follow upon showing that the Galois closure of $Z \to \mathbb{P}^{1}_{x} \to \mathbb{P}^{1}_{T}$ is $D_8$. We write $W$ to denote the discriminant curve (see \cite[Lemma A.7]{DGKM} or  \cite[\S 2.2]{donagi}) of $Z \to \mathbb{P}^{1}_{T}$. It suffices to show that $W \to \mathbb{P}^{1}_{T}$ is not the trivial cover. By using \cite[Lemma 2.3]{donagi}, we deduce that $W \to \mathbb{P}^{1}_{T}$ is branched at $2$ distinct points, namely $T(w_1)$ and $T(w_2)$ where $w_1$ and $w_2$ are the two roots of the remaining quadratic factor $F_3(x)$. It's easy to check that $T(w_1) \neq T(w_2)$ (as otherwise either $Z$ itself or the quadratic splitting would be singular). This completes the proof of (2). Claim (3) follows from the proof of (2).
 \qedhere \end{proof}

\begin{notation}We write $C_2^a$ and $C_2^b$ to denote the two non-central subgroups of $D_8$ of order $2$. In addition, write $(C_2^2)^a$ and $(C_2^2)^b$ to denote the two subgroups of $D_8$ which contain $C_2^a$ and $C_2^b$ respectively. \end{notation}
\begin{theorem} \label{pseudo_brauer_richelot}
Let $Z/K$ be a $C2D4$ curve which admits a non-singular, quadratic splitting. Then, the pseudo Brauer relation $C_2^a-C_2^b+(C_2^2)^b-(C_2^2)^a$ for $D_8$ and $\Gamma_d$ verifies a Richelot isogeny $\textup{Jac}_{Z} \to \textup{Jac}_{\tilde{Z}_{d}}$.
\end{theorem}
\begin{proof} We note that $C_2^a-C_2^b+(C_2^2)^b-(C_2^2)^a$ is a Brauer relation for the group $D_8$. Following the notation from Proposition \ref{automorphisms_on_richelot}, we deduce that $(C_2^2)^a$ and $(C_2^2)^b$ coincide with the subgroups $\langle s,r^2 \rangle$ and $\langle sr, r^2\rangle.$ By Proposition \ref{automorphisms_on_richelot}(3), the quotients $\Gamma_d/(C_2^2)^a$ and $\Gamma_d/(C_2^2)^b$ are curves of genus $0$. This gives the required result.
\end{proof}
We write $\epsilon_{r}$ the denote the non-identity, $1$-dimensional character of $D_{8}$ on which $r$ acts trivially and $\rho$ for the irreducible $D_8$-representation of dimension $2$. 
\begin{lemma} \label{rep_strucutre} The $D_{8}$-representations  
\[ \mathcal{X}_{2}(\textup{Jac}_{\Gamma_d}) \ \ \textup{and} \ \  \epsilon_{r}^{\textup{rk}_{2 }\textup{Jac}_{\Gamma_d/\langle r^2 \rangle}} \oplus \rho^{\textup{rk}_{2}  \textup{Jac}_{Z}}\] become isomorphic after extending scalars to $\mathbb{C}$.
\end{lemma}
\begin{proof}
The set of irreducible $\mathbb{C}$-representations for $D_8$ are $\{ \mathbf{1} ,\epsilon_{s}, \epsilon_{r}, \epsilon_{sr}, \rho \}$. Let $(C_2^2)^a= \langle s, r^2 \rangle$. Then, by Frobenius reciprocity, $\langle \mathcal{X}_{2} , \mathbb{C}[D_8/(C_2^2)^a]  \rangle = \langle \mathcal{X}_{2} , \mathbf{1} \oplus \epsilon_{s}  \rangle $ coincides with the dimension of $(C_2^2)^a$-invariant vectors in $\mathcal{X}_{2}$. This coincides with the $2^{\infty}$-Selmer rank of the Jacobian of the quotient of $\Gamma_d$ by $(C_2^2)^a$ (see \cite[Remark A.29]{DGKM}), which must be $0$ by Proposition \ref{automorphisms_on_richelot}(3). Therefore, $\langle \mathcal{X}_{2} , \mathbf{1} \rangle = \langle \mathcal{X}_{2} ,  \epsilon_{s}  \rangle =0$. The multiplicities of the remaining irreducible representations in $\mathcal{X}_{2}$ are found in a similar way. \qedhere
\end{proof}

\begin{theorem} \label{local_formula_Richelot} Let $Z/K$ be a $C2D4$ genus $2$ curve which admits a non-singular quadratic splitting. Then,
\[\textup{rk}_{2}(\textup{Jac}_{Z}/K)= \sum_{v \ \textup{place of} \ K} \textup{ord}_{2} \ \Lambda_{\Theta}( \Gamma_d/ K_{v}) \mod 2,\]
where $\Theta$ is the pseudo Brauer relation for $D_8$ and $\Gamma_d$ afforded by Theorem \ref{pseudo_brauer_richelot}.
\end{theorem}

\begin{proof}
We write $\mathcal{X}_{2}$ to denote $\mathcal{X}_{2}(\textup{Jac}_{\Gamma_d}).$ An elementary computation reveals that $\mathcal{C}_{\Theta}(\rho)\equiv 2$, while $\mathcal{C}_{\Theta}(\epsilon_{r})\equiv 1$. Therefore,  $ S_{\Theta,2} = \{ \rho \}$ (see Definition \ref{def_special_subset}). By Lemma \ref{rep_strucutre}, we deduce that $ \sum_{\tau \in \mathcal{S}_{\Theta,2}} \langle \tau , \mathcal{X}_{2} \rangle = 2^{\textup{rk}_2\textup{Jac}_{Z}}.$ The result follows upon applying Theorem \ref{regulator_constant_rank_parity} to the pseudo Brauer relation $\Theta$ for $D_8$ and $\Gamma_d$ afforded by Theorem \ref{pseudo_brauer_richelot} with $p=2$. \qedhere
\end{proof}
\section{Prym varieties of trigonal curves}\label{Sec_Prym}

We let $Z/K$ be a geometrically connected curve of genus $g$, and in addition we suppose $\textup{Jac}_{Z}[2](k)$ has a non-trivial torsion point, or equivalently there exists an unramified double cover $ \pi: \tilde{Z} \to Z$ defined over $K$. The \textit{Prym variety} denoted $\textup{Prym}(\tilde{Z}/Z)$ of $\pi$ is the abelian variety given by the connected component of $\textup{Ker}(\pi_{*})$ containing the identity. It is classically well-known that there exists a $K$-rational isogeny 
\begin{equation*} \textup{Jac}(\tilde{Z}) \to \textup{Jac}(Z) \times \textup{Prym}(\tilde{Z}/Z),\end{equation*}
It is worthy of note that $\textup{Prym}(\tilde{Z}/Z)$ need not be a Jacobian, but it necessarily is principally polarised. The trigonal construction \cite[\S 2.4]{donagi} establishes that for trigonal curves $Z$ defined over algebraically closed fields, the Prym variety is a Jacobian of a tetragonal curve. A generalisation of this in the case of non-algebraically closed fields can be found in \cite{bruin-prym} where the authors prove that the Prym is necessarily a Jacobian upon taking a quadratic twist of the cover $\pi:\tilde{Z} \to Z$ if needed.

\begin{definition} \label{defn_quadratic_twist} Let $\pi: \tilde{Z} \to Z$ be a degree $2$ cover. We say that $\pi_{d}:\tilde{Z}_{d} \to Z$ is a quadratic twist of $\pi$ by $d \in K^{\times}$ if there exists an isomorphism $\psi:\tilde{Z}_{d} \xrightarrow{\sim} \tilde{Z}$ defined over $K(\sqrt{d})$ such that $\pi_{d}= \pi \circ \psi.$
\end{definition}

As can be deduced from \cite[Proposition 2.3]{bruin-prym}, for an appropriate value of $d \in K^{\times}$, there exists an isogeny 
\begin{equation*} \label{prym_isogeny} \textup{Jac}_{\tilde{Z}_{d}} \to \textup{Jac}_{Z} \times \textup{Prym}(\tilde{Z}_{d}/Z), \end{equation*} and, at this instance, $\textup{Prym}(\tilde{Z}_{d}/Z)$ is a Jacobian of a tetragonal curve. We note that we allow $d$ to lie in $K^{\times 2}$ in which case $\tilde{Z}_{d} \cong_{K} \tilde{Z}$ is the trivial twist. We write $X_{\tilde{Z}} \to \mathbb{P}^{1}$ to denote the Galois closure of $\tilde{Z} \to Z \to \mathbb{P}^{1}$, and we define $X_{\tilde{Z}_{d}}$ similarly.

\begin{theorem} \label{pseudo_brauer_verifiability_prym} Let $\pi:\tilde{Z} \to Z$ be an unramified degree $2$ cover. Then, there exists $d \in K^{\times}$, and a quadratic twist $\pi_d:\tilde{Z}_{d} \to Z$ such that $S_4$ acts on $X_{\tilde{Z}_{d}}$. In addition, $\Theta = C_2^2 -  D_8 - S_3 + S_4$ is a pseudo Brauer relation for $S_4$ and $X_{\tilde{Z}_{d}}$ which verifies an isogeny $\textup{Jac}_{\tilde{Z}_{d}} \to \textup{Jac}_{Z} \times \textup{Prym}(\tilde{Z}_{d}/Z)$.
\end{theorem}
\begin{proof}
In general, $\textup{Gal}(X_{\tilde{Z}} \to \mathbb{P}^{1})  \cong C_2^3 \rtimes S_3 \cong C_2 \times S_4$. As can be deduced from \cite{donagi}, over $\overline{K}$, we have $\textup{Gal}(X_{\tilde{Z}}/\overline{K} \to \mathbb{P}^{1}/\overline{K})\cong S_4$. Therefore, we deduce that $X_{\tilde{Z}}$ need not be geometrically connected, and its geometric components are (at worst) defined over a quadratic extension $K(\sqrt{d})$. Upon twisting $\pi$ by $d \in K^{\times}$ however, we have that $\textup{Gal}(X_{\tilde{Z}_{d}}/ \mathbb{P}^{1})=S_4$ in which case $X_{\tilde{Z}_{d}}$ is geometrically connected. By replacing  $X_{\tilde{Z}_{d}}$ with the $S_4$-closure of $\tilde{Z}_{d} \to Z \to \mathbb{P}^{1}$ in the sense of \cite[Definition A.16]{DGKM} we can assume that the $\textup{Gal}(X_{\tilde{Z}_{d}}/ \mathbb{P}^{1})$ is always $S_4$. Then, by construction, we have that 
\begin{equation} \label{quotients_prym} X/S_4 = \mathbb{P}^{1}, X_{\tilde{Z}_{d}}/D_8 = Z \ \textup{while} \ X_{\tilde{Z}_{d}}/(C_2^2) = \tilde{Z}_{d}. \end{equation} 
We have that $C_2^2 - D_8 - S_3 +S_4$ is a Brauer relation for the group $S_4$. Therefore, it's a pseudo Brauer relation for $S_4$ acting on $X_{\tilde{Z}_{d}}$ and in view of the quotients given in \eqref{quotients_prym}, we can deduce an isogeny $\textup{Jac}_{\tilde{Z}_{d}} \to \textup{Jac}_{Z} \times \textup{Jac}_{X_{\tilde{Z}_{d}/S_3}}$. As detailed in \cite[\S 2.2]{bruin-prym} and \cite[Theorem 2.11]{donagi}, the Jacobian of the quotient $X_{\tilde{Z}_{d}}/S_3$ is the Prym variety of $\pi_{d}$. This completes the proof.
\end{proof}

For the remainder of this section, we write $\sigma$ to denote the standard $3$-dimensional irreducible representation of $S_4$, and $\tau$ for the irreducible $2$-dimensional representation lifted from the $S_3$ quotient of $S_4$. 
\begin{lemma} \label{lemma_multiplicities} For any prime $p$, the multiplicities of $\sigma$ and $\tau$ in $\mathcal{X}_{p}(\textup{Jac}_{X_{\tilde{Z}_{d}}})$ satisfy
\[\langle \mathcal{X}_{p}(\textup{Jac}_{X_{\tilde{Z}_{d}}}), \sigma \rangle = \textup{rk}_{p}\ \textup{Prym}(\tilde{Z}_{d}/Z) \ \ \textup{and} \ \ \langle \mathcal{X}_{p}(\textup{Jac}_{X_{\tilde{Z}_{d}}}), \tau \rangle = \textup{rk}_{p}\ \textup{Jac}_{Z}.  \]
\end{lemma}
\begin{proof}
Follows in exactly the same way as Lemma \ref{rep_strucutre}.
\end{proof}

\begin{theorem} \label{prym_local_formula}
Let $\pi: \tilde{Z} \to Z$ be an unramified,double cover of a trigonal curve $Z$. Then, there exists a $d \in K^{\times}$ and a quadratic twist $\pi_d : \tilde{Z}_{d} \to Z$ of $\pi$ such that
\[ \textup{rk}_{2} \ \textup{Jac}_{Z} + \textup{rk}_{2} \ \textup{Prym}(\tilde{Z}_{d}/Z) \equiv \sum_{v \ \textup{place of} \ K} \textup{ord}_{2} \ \Lambda_{\Theta}(X_{\tilde{Z}_{d}}/K_v) \mod 2, \]
where $\Theta$ is the pseudo Brauer relation for $S_4$ and $X_{\tilde{Z}_{d}}$ afforded by Theorem \ref{pseudo_brauer_verifiability_prym}.
\end{theorem}
\begin{proof}
We write $\mathcal{X}_2$ to denote $\mathcal{X}_{2}(\textup{Jac}_{X_{\tilde{Z}_{d}}})$. Then, we claim that $S_{\Theta,2} = \{\tau, \sigma \}.$ Indeed, by combining the proof of Theorem \ref{pseudo_brauer_verifiability_prym} with Lemma \ref{trivial_regulator_constant}(1), we deduce that $S_{\Theta,2} \subseteq \{\tau, \sigma \}$. An elementary computation reveals that $\mathcal{C}_{\Theta}(\tau) \equiv \mathcal{C}_{\Theta}(\sigma) \equiv 2$, and therefore $S_{\Theta,2} = \{\tau, \sigma\}$. In view of Lemma \ref{lemma_multiplicities}, the local formula asserted by the claim follows upon applying Theorem \ref{regulator_constant_rank_parity} to $\Theta$ afforded by Theorem \ref{pseudo_brauer_verifiability_prym} with $p=2$.
\end{proof}
\section{Elliptic curves with cyclic isogeny}\label{Sec_EC}
In this section, we deduce a local expression for the $p^{\infty}$-Selmer rank of an elliptic curve that admits a cyclic isogeny of prime degree $p$ using pseudo Brauer relations and regulator constants. For the remainder of this section, we fix $E/K$ to be an elliptic curve, and we let $P\in E[p]$ be a generator for the kernel of this isogeny. Finally, we write $K(P)$ to denote the field acquired by adjoining the coordinates of $P$ to $K$ and $\tau_{P}$ for the translation-by-$P$ automorphism on $E$.

\begin{lemma}
\label{cyclicextension} 
The field extension $K(P)/K$ is Galois, and the homomorphism
\begin{equation*}
\begin{split} 
\label{eqn:Galois} \textup{Gal}{(}K(P)|K{)} &\xrightarrow{\psi} \textup{Aut}(\langle \tau_{P} \rangle)  \\
\sigma &\mapsto (\tau_{P} \mapsto \tau_{P^{\sigma}})
\end{split}
\end{equation*} defines a faithful action on $\langle \tau_{P} \rangle$.
\end{lemma}
\begin{proof} 
Upon fixing a basis for $E[p] = \langle P , Q \rangle$, we can identify $\textup{Gal}(K(E[p])/K)$ with a subgroup of $GL_{2} (\mathbb{F}_{p})$. Then, $K=K(E[p])^{B}$ while $K(P)=K(E[p])^{B'}$, where $B$ denotes the subgroup of upper triangular matrices, while $B'$ denotes the subgroup of $B$ with top, left entry equal to $1$. Since $B'$ is a normal subgroup of $B$, then $K(P)/K$ is Galois. It suffices to prove that $\psi$ is injective. Indeed, if $\psi(\sigma_1)=\psi(\sigma_2)$, then $\sigma_1\sigma_2^{-1}$ fixes $P$ pointwise and therefore $\sigma_1=\sigma_2$. \qedhere

\end{proof}

\begin{notation}
For the remainder of this section, we write $M=K(P)$. Then, $\textup{Gal}(M/K) $ acts on $M(E)$ (on the left) via its action on constant functions $M $. For $f \in M(E)$,  we denote this action exponentially on the left by $f \mapsto {}^{\sigma} f$. By Lemma \ref{cyclicextension}, $\textup{Gal}(M/K) \cong C_m$ is cyclic of order dividing $p-1$, and in addition we can form the semi-direct product $ \langle \tau_{P} \rangle \rtimes \textup{Gal}(M/K) = C_{p} \rtimes C_m$. Finally, we write $_{K}(E_M)$ to denote the curve obtained from $E_M= E \times_{K} M$ by forgetting the $M$-structure (see \S\ref{notation_for_curves}).

\begin{proposition} \label{G-action:ellipticCurves} The semi-direct product $G= C_{p} \rtimes C_m$ acts on ${}_{K}(E_{M})$ via $K$-automorphisms. In addition, 
\begin{equation} \label{quotients_elliptic_curves} {}_{K}(E_{M})/C_m= E/K, \ \  {}_{K}(E_{M})/C_{p}= {}_{K}(E'_{M}), \ \ {}_{K}(E_{M})/G = E'/K,  \end{equation}  where $E'=E/\langle P \rangle$ denotes the $p$-isogenous elliptic curve.

\end{proposition}
\begin{proof}
The group $ \langle \tau_{P} \rangle \rtimes \textup{Gal}(M/K) = C_{p} \rtimes C_m$ acts on functions in $M(E)$ via 
\[ f \xrightarrow{( \tau, \sigma )} {}^{\sigma} f \circ \tau. \]  By using \cite[Exercise 7.8.2]{Diamond-Shurman}, this fits into a Galois diagram

\begin{figure}[H]
\begin{center}\begin{tikzpicture}
    \node (Q1) at (0,0) {$K(E')$};
    \node (Q2) at (-1.5,1.3) {$M(E')$};
    \node (Q3) at (+1.5,1.7) {$K(E)$};
    \node (Q4) at (0,3) {$M(E)$};
    \draw (0.2,2.8)--node[above]{ \hspace{1cm}\tiny{$\textup{Gal}(M/K)$}}(1.4,1.9);
        \draw (-0.2,2.8)--node[above]{ \hspace{-0.7cm}\tiny{$\langle \tau_{P} \rangle$}}(-1.4,1.6);
    \draw (Q3)--(Q1);
    \draw (Q2)--(Q1);
    \end{tikzpicture} \end{center}\end{figure}  
\noindent By \cite[Remark A.31]{DGKM}, we deduce that the group $ C_{p} \rtimes C_m$ acts on ${}_{K}(E_{M})$ via $K$-automorphisms as required. In addition, for any subgroup $H \leq C_p \rtimes C_m$, the field of $H$-invariant functions in the function field of ${}_{K}(E_M)$ coincides with the function fields of the quotient curves $_{K}(E_M)/H$ in view of \cite[Remark A.30]{DGKM}. The claim regarding the various quotients follows by combining this observation with the Galois diagram above.
\end{proof}
\end{notation}
\begin{remark} \label{remark:field of definition}
For a given subgroup $D \leq C_{p} \rtimes C_m$, the geometric components of the quotient curve ${}_{K}(E_{M})/D$ are defined over  $M^{\pi(D)}$ where $\pi : C_{p} \rtimes C_m \to C_m $ denotes the natural map. Indeed, given an element $ (\tau, \sigma )$ in $C_{p} \rtimes C_m$, then $\tau$ has no effect on {constant} functions, while by construction $\sigma$ acts only on constant functions $M$ in $M(E)$. Therefore, $M(E)^{D} \cap \overline{K} = M(E)^{D} \cap M = M^{\pi(D)}$.  From this, we deduce that the geometric components of ${}_{K}(E_{M})/D$ are defined over $ M^{\pi(D)}$.

\end{remark} 

\begin{theorem} \label{pseudo_brauer_verifiability_EC} Let $G= C_{p} \rtimes C_m$ be the subgroup of $K$-automorphisms of ${}_{K}(E_M)$ afforded by Proposition \ref{G-action:ellipticCurves}. Then, $\Theta=C_m-G$ is a pseudo Brauer relation for $C_{p} \rtimes C_m$ and ${}_{K}(E_M)$ verifying an isogeny $E \to E'$.
\end{theorem} 

\begin{proof}
We write $V_{\ell}$ to denote the $G$-representation of the $\ell$-adic Tate module of the Jacobian of ${}_{K}(E_M)$. We decompose $\mathbb{C}[G/C_m] \cong \mathbb{C}[G/G] \oplus T,$ where $T$ is some complementary $G$-representation. Then, $\langle V_{\ell} ,T \rangle$ satisfies
\begin{equation}
 \label{morally-zeroness}
\langle V_{\ell}, T\rangle 
\underset{\mathrm{Frob. rec.}}{=}\textup{dim} \ V_{\ell}^{C_m}  - \textup{dim} \ V_{\ell}^{G}  \underset{ \eqref{quotients_elliptic_curves}}{=}\textup{dim} \ V_{\ell}(E)-\textup{dim} \ V_{\ell} (E')=0.
\end{equation} Therefore, $\Theta=C_m-G$ is a pseudo Brauer relation for $C_{p} \rtimes C_m$ and ${}_{K}(E_M)$. By invoking \eqref{quotients_elliptic_curves} again, the claim regarding pseudo Brauer verifiability follows. \qedhere
\end{proof}
We now find a description for the $G$-module structure on $\Omega^{1}$ of the Jacobian of ${}_{K}(E_M)$. Using \cite[Lemma A.22]{DGKM}, we deduce that $\textup{Jac}_{{}_{K}(E_M)}=\textup{Res}_{M/K} E_M$. 

\begin{lemma} \label{regular_represenation_omega_1} 
There exists an isomorphism of $C_m = \textup{Gal}(M/K)$-modules,
\[\Omega^{1}(\textup{Res}_{M/K}E_{M}) \otimes \mathbb{C} \cong \textup{Ind}_{\{e\}}^{C_m} \mathbf{1}. \]
\end{lemma}
\begin{proof}
As an abelian variety over $M$, $\textup{Res}_{M/K} E_{M}$ decomposes into a product $\prod_{g } E^{g},$ where $g$ runs over the elements of $\textup{Gal}(M/K).$ Then,  $\textup{Gal}(M/K)$ acts on $\textup{Res}_{M/K} E_{M}$, and the corresponding action on $\prod_{g } E^{g}$ permutes the components of the product. In particular, $h \in \textup{Gal}(M/K)$ acts via mapping $E^{g} \to E^{g h}$. Therefore, $\Omega^{1}(\textup{Res}_{M/K}E_{M}) \otimes M \cong \textup{Ind}_{\{e\}}^{C_m} \mathbf{1}$.
\end{proof}
\begin{proposition}\label{prop:tatemodule}For any prime number $\ell$, the $G$ representations
\[V_{\ell}(\textup{Res}_{M/K}{E_{M}}) \ \ \textup{and} \ \ (\textup{Ind}_{C_p}^{G} \mathbf{1} )^{\oplus 2} \] are isomorphic representations after extending scalars to $\mathbb{C}$.  In addition, \[\Omega^{1}(\textup{Res}_{M/K}{E_{M}})\otimes \mathbb{C} \cong \textup{Ind}_{C_p}^{G} \mathbf{1}.\] 
\end{proposition}
\begin{proof}
For brevity, we write $A$ to denote $\textup{Res}_{M/K}E_{M}$ and $V_{\ell}$ (resp. $\Omega^{1}$) for the G-representations obtained from $V_{\ell}(A)$ (resp. $\Omega^{1}(A)$) after extending scalars to $\mathbb{C}$. By \cite[Lemma 3.18]{DGKM}, $V_{\ell}$ satisfies $ \textup{dim}(V_{\ell})^{D} = 2\cdot \textup{dim}(\textup{Jac}_{_{K}(E_{M})/D})$ for any subgroup $D$ of $G$. By remodelling the argument presented in \cite[Theorem 3.3]{etalecohomology_of_hyperelliptic_curves}, we deduce that $V_{\ell}$ is the unique representation satisfying the following.
\begin{enumerate}
\item all character values of $G$ on $V_{\ell}$ are rational
\item and $ \textup{dim}(V_{\ell})^{D} = 2\cdot \textup{dim}(\textup{Jac}_{_{K}(E_{M})/D})$.
\end{enumerate}
 The dimension of $D$-invariant vectors of $\textup{Ind}_{C_p}^{G} \mathbf{1}$ is 
\begin{align*}
\textup{dim}\ (\textup{Ind}_{C_p}^{G} \mathbf{1})^{D} &= \langle \textup{Ind}_{C_p}^{G} \mathbf{1} ,\textup{Ind}_{D}^{G}\mathbf{1}\rangle = \langle \mathbf{1}, \textup{Res}_{C_{p}}^{G} \textup{Ind}_{D}^{G} \mathbf{1} \rangle 
= \# D \backslash G / C_p.
\end{align*}
In addition, we deduce that $\# D \backslash G / C_p = \frac{\#\textup{Gal}(M/K)}{\#\pi(D)} = [M^{\pi(D)}:K]$ which agrees with the dimension of the Jacobian of ${}_{K}(E_M)/D$ in view of Remark \ref{remark:field of definition}. Since $(\textup{Ind}_{C_p}^{G} \mathbf{1})^{\oplus 2}$ is realisable over $\mathbb{Q}$, both conditions listed are satisfied for $(\textup{Ind}_{C_p}^{G} \mathbf{1})^{\oplus 2}$. By uniqueness, $V_{\ell} \cong (\textup{Ind}_{C_p}^{G} \mathbf{1})^{\oplus 2}$ as required. 
For the claim regarding $\Omega^{1}$, we begin by decomposing $V_{\ell}$ into irreducible representations, say
\[V_{\ell} \cong \bigoplus_{i=1}^{m} \tilde{\chi}_{i}^{\oplus 2}, \]
where each $\tilde{\chi}_{i}$ is a lift of a $1$-dimensional character $\chi_{i}$ from the $G/C_{p}\cong C_{m}$ quotient. By \cite[Lemma 5.1(1)]{DGKM}, we deduce that $\Omega^{1}$ is a subrepresentation of $V_{\ell}$ and therefore, we can decompose $\Omega^{1} \cong \bigoplus_{i \in I} \tilde{\chi_i}^{a_i}$, for a subset $I \subseteq \{1,2,\dots, m\}$ and $a_{i} \in \{0,1,2\}$. By Lemma \ref{regular_represenation_omega_1}, $\textup{Res}_{C_m}^{G} \Omega^{1} \cong \textup{Ind}_{\{e\}}^{C_m} \mathbf{1} \cong \bigoplus_{i=1}^{m} \chi_{i}$. Since  $\textup{Res}_{C_m}^{G} \tilde{\chi}_{i} = \chi_{i}$ we deduce that $I = \{1,2,\dots,m\}$, and $a_i=1$ for all $i \in I$. Therefore,
$\Omega^{1} \cong \oplus_{i=1}^{m} \tilde{\chi}_i \cong \textup{Ind}_{C_p}^{G} \mathbf{1}$ as required.
\qedhere
\end{proof}

\begin{theorem} \label{local_formula_EC}Let $E/K$ be an elliptic curve which admits a $K$-rational cyclic isogeny of prime degree $p$. Then,

\[\textup{rk}_{p}(E/K)= \sum_{v \ \textup{place of} \ K} \textup{ord}_{p} \ \Lambda_{\Theta}({}_{K}(E_{M}) / K_{v}) \mod 2 ,\]
where $\Theta$ is the pseudo Brauer relation for $C_{p} \rtimes C_m$ and ${}_{K}(E_{M})$ afforded by Theorem \ref{pseudo_brauer_verifiability_EC}.\end{theorem}
\begin{proof} We write $\mathcal{X}_{p}$ to denote $\mathcal{X}_{p}(\textup{Res}_{M/K}E_{M})$. We show that $S_{\Theta,p} = \{ \mathbf{1} \}$. By choosing the trivial pairing on $\mathbf{1}$, we deduce that $\mathcal{C}_{\Theta}(\mathbf{1}) \equiv \frac{m}{pm}\equiv p \mod \mathbb{Q}^{\times 2}$, from which we deduce $\{1\} \subseteq S_{\Theta,p}$. By Lemma \ref{trivial_regulator_constant}(1), showing equality will follow upon showing that $\mathcal{X}_{p}$ and $\mathbb{C}[G/C_m]$ share no common irreducible constituents other than $\mathbf{1}$. This follows as in \eqref{morally-zeroness}. Finally, the local expression asserted by the claim follows upon applying Theorem \ref{regulator_constant_rank_parity} to $\Theta$ afforded by Theorem \ref{pseudo_brauer_verifiability_EC}. We note that Theorem \ref{regular_represenation_omega_1} is applicable in view of Proposition \ref{prop:tatemodule} which proves that $\Omega^{1}$ is self-dual as a $G$-representation. \qedhere
\end{proof}
\section{Genus $2$ curves with elliptic subcovers}\label{Sec_genus2}

\subsection{Genus $2$ curves admitting a map to an elliptic curve}\label{subsec_EC}
In this section, we provide a local expression for the parity of the $p^{\infty}$-Selmer rank for the Jacobian of a genus $2$ curve $Y/K$ which admits an elliptic subcover of prime degree $p$, equivalently a $K$-rational morphism $\phi: Y \to E$  to an elliptic curve $E/K$. As a result, its Jacobian, $\textup{Jac}_{Y}$, is isogenous to $E \times E'$ for some complementary elliptic curve $E'$. As noted in \cite{Shioda}, the complimentary elliptic curve $E'$ is not uniquely determined.

This section is designed to accomplish three distinct goals. Firstly, we establish a canonical choice for the complementary elliptic curve, referred to as $E'$, under specific generality assumptions regarding the parameter $\phi$

Under some generality assumptions on $\phi$ (see Definition \ref{def:generic} below), we provide a canonical choice for the complementary elliptic curve $E'$, then we prove pseudo Brauer verifiability for the isogeny $\textup{Jac}_{Y} \to E \times E'$, and finally we use regulator constants to obtain a local expression for the parity of the $p^{\infty}$-Selmer rank of $\textup{Jac}_{Y}$. A central part of our argument is the following commutative diagram
 \begin{center}

\begin{tikzcd}
Y \arrow[r, "\phi"] \arrow[d, "\pi_{Y}"] & E \arrow[d, "\pi_{E}"] \\
\mathbb{P}_{x}^{1} \arrow[r, "\Phi"]   & \mathbb{P}_{u}^{1}         
\end{tikzcd}
\end{center}
where $\pi_Y$ and $\pi_E$ are the natural hyperelliptic covers of $Y$ and $E$, and $\Phi$ is the unique map making this diagram commute. As in \S\ref{General_notation}, we write $\mathbb{P}^{1}_{x}$ and $\mathbb{P}^{1}_{u}$ to indicate the choice of a parameter $x$ and $u$ on $\mathbb{P}^{1}/K$.

\subsection{Elliptic subcovers of even prime degree}
For ease of exposition, we first illustrate how we utilise this commutative diagram in the case when $p=2$.

\begin{proposition} \label{elliptic_involution}
Let $\phi: Y \to E$ be a degree $2$ morphism induced by the involution $r$ on $Y$. Writing $h_{Y}$ to denote the hyperelliptic involution, then $C_2 \times C_2 \cong \langle r , h_{Y} \rangle $ acts on $Y$ via $K$-automorphisms. This fits into a Galois diagram

\begin{center}
\begin{tikzcd}
                                                    & Y \arrow[ld, "\phi"', no head] \arrow[d, no head] \arrow[rd, "\pi_Y", no head] &                                                                  \\
E=Y/\langle r \rangle \arrow[rd, "\pi_E"', no head] & Y/\langle rh_Y \rangle \arrow[d, no head]                                   & \mathbb{P}^{1}=Y/\langle h_Y \rangle \arrow[ld, "\Phi", no head] \\
                                                    & {\mathbb{P}^{1}=Y/\langle r, h_Y \rangle}                                      &                                                                 
\end{tikzcd}
\end{center}
 In addition, $\{e\}+2(C_2\times C_2) - \langle r \rangle-\langle r h_Y \rangle-\langle h_Y \rangle$ is a pseudo Brauer relation for $C_{2}\times C_{2}$ and $Y$ which verifies an isogeny \[\textup{Jac}_{Y} \to E \times \textup{Jac}_{Y/\langle rh_{Y} \rangle}.\] Finally, we have
 \[\textup{rk}_{2}(\textup{Jac}_{Y}) \equiv \sum_{v \ \textup{place of} \ K} \textup{ord}_{2} \ \Lambda_{\Theta}(Y/K_v) \mod 2. \]
\end{proposition}
\begin{proof}
Since $h_{Y}$ lies in the centre of $\textup{Aut}_{K}(Y)$ we deduce that $r$ and $h_{Y}$ commute. Therefore, $\langle r , h_{Y} \rangle \cong C_2 \times C_2$, and in addition $E=Y / \langle r \rangle$.  Since $\{e\}+2(C_2\times C_2) - \langle r \rangle - \langle rh_{Y} \rangle - \langle h_Y \rangle $ is a Brauer relation for the group $C_{2} \times C_{2}$, and $Y/\langle r, h_{Y} \rangle$ and $Y/\langle h_{Y} \rangle$ are both curves of genus $0$, then the result concerning pseudo Brauer verifiability follows. The local formula follows in exactly the same way as Theorem \ref{local_expression_KT}.
\end{proof}
In view of Proposition \ref{elliptic_involution}, we deduce that a canonical choice for the complimentary elliptic curve $E'$ is $Y/\langle rh_Y\rangle$ in the case when $p=2$. It is worthy of note that Theorem \ref{complementary_elliptic_curve} below provides a generalisation of Proposition \ref{elliptic_involution}.

\subsection{Elliptic subcovers of odd prime degree}
In the following subsection, we acquire a Galois cover of curves $X_Y \to \mathbb{P}_{u}^{1}$ in the presence of an elliptic subcover of odd prime degree. We only give the construction in the case where $\phi$ is generic (following Kuhn's terminology \cite{Kuhn}).
\begin{definition} \label{def:generic} We say that an elliptic subcover $\phi: Y \to E$ is \textit{generic} if it is unramified over $E[2]$, i.e. the Weierstrass points of $E$. 
\end{definition}

\begin{proposition} \label{Galois_closure} Let $Y/K$ be a genus $2$ curve with a $K$-rational elliptic subcover $\phi:Y \to E$ of odd prime degree $p$. In addition, suppose $\phi$ is generic. Then, the Galois closure $X_{Y}\to \mathbb{P}_{u}^{1}$ of $Y \xrightarrow{\phi} E \xrightarrow{\pi_E} \mathbb{P}_{u}^{1}$ is geometrically connected and has Galois group $S_{p} \times C_{2}$.  In addition, these fit into a Galois diagram

 \begin{figure}[h!]
\begin{center}\begin{tikzpicture}
    \node (Q1) at (2,-2) {$\mathbb{P}_{u}^{1}=X_Y/(S_{p}\times C_2)$};
    \node (Q2) at (-2.0,1.5) {$Y=X_Y/(S_{p-1} \times \{e\})$};
    \node (Q3) at (+2.0,-0.5) {$E=X_Y/(S_p \times \{e\})$};
    \node (Q4) at (0.5,3) {$X_Y$};
    \node (Q5) at (-2.0,-0.0) {$\mathbb{P}_{x}^{1}=X_Y/(S_{p-1}\times C_2)$};
    \draw (Q4)--(Q2);
    \draw (Q4)--(Q3);
    \draw (Q3)--node[right=1.0mm] {$\hspace{0mm}\pi_E$} (Q1);
    \draw (Q2)--node[left=0.0mm] {$\hspace{-3mm}\pi_Y$}(Q5);
    \draw (Q1)--node[below=1.0mm] {$\hspace{-5mm}\Phi$} (Q5);
    \draw (Q3)--node[above=0.5mm] {$\hspace{0mm}\phi$}(Q2);
    \end{tikzpicture} \end{center}\end{figure}
\end{proposition}
\begin{proof}
We write $\{d_1,d_2,d_3,d_4\}$ to denote the branch locus of $E \to \mathbb{P}^{1}_{u}$, i.e. $u(E[2])$. By \cite[\S 1]{Kuhn}, the branch locus of $\Phi: \mathbb{P}_{x}^{1} \to \mathbb{P}_{u}^{1}$  consists of $\{d_1,d_2,d_3,d_4,d_5\}$ when $p \geq 5 $ and $\{d_1,d_2,d_3,d_5\}$ where $d_4 \neq d_5$ when $p=3$. Write $X_{\mathbb{P}^{1}} \to \mathbb{P}^{1}_{u}$ to denote the Galois closure of $\mathbb{P}^{1}_{x} \to \mathbb{P}^{1}_{u}$. By \cite[\S 1]{Kuhn}, the ramification structure of $\Phi$ (see Definition \ref{Def:ramification_structure}) in the generic case is $\{ (2)^{\frac{p-1}{2}},(2)^{\frac{p-1}{2}},(2)^{\frac{p-1}{2}},(2)^{\frac{p-3}{2}}, (2)^{1}\}$ for $ p \geq 3$. We now fix an arbitrary embedding $K \hookrightarrow \mathbb{C}$. Following the notation from Theorem \ref{theorem_miranda}, we write $\sigma_i$ to denote the  monodromy of $X_{\mathbb{P}^{1}} \to \mathbb{P}^{1}_{u}$ at $d_i$. It follows that $\sigma_5$ is a transposition, and therefore the Galois group of $X_{\mathbb{P}^{1}} \to \mathbb{P}^{1}_{u}$ viewed as a morphism of complex curves is the whole of $S_p$. Therefore, the corresponding $K$-rational morphism $X_{\mathbb{P}^{1}} \to \mathbb{P}^{1}_{u}$ must have Galois group $S_p$.

Showing that the Galois group of $X_{Y} \to \mathbb{P}^{1}$ is $S_{p} \times C_{2}$ (as opposed to $S_{p})$ will follow upon showing that the discriminant curve $W = {X}_{\mathbb{P}^{1}}/A_p$ (see \cite[Lemma A.7]{DGKM}) isn't isomorphic to $E$ over $\bar{K}$. It suffices to show that the branch points of $W \to \mathbb{P}^{1}_{u}$ do not coincide with $u(E[2]).$ In particular, if $ p \equiv 1 \mod 4$, then $\sigma_{1}, \sigma_{2}, \sigma_{3}$ are  even permutations, while $\sigma_{4}$ and $\sigma_{5}$ are odd permutations. Therefore, the branch locus of $D \to \mathbb{P}_{u}^{1}$ is $\{d_4,d_5\}$. Similarly, if $p\equiv 3 \mod 4$, then the corresponding branch locus $W \to \mathbb{P}^{1}$ is $\{d_1,d_2,d_3,d_5\}$. Therefore, $W \not\cong  E$ as required. 

We show that $X_Y$ is geometrically connected.  From above, $X_{\mathbb{P}^{1}}$ must be geometrically connected, while the function field $K(X_Y)$ contains $K(X_{\mathbb{P}^{1}})$ as a degree $2$ subfield. Now aiming for a contradiction, suppose that $X_Y$ is not geometrically connected. Under this assumption, the geometric components of $X_Y$ must be defined over a quadratic field extension of $K$, say $K(\sqrt{m})$. Therefore, the $3$ quadratic extensions of $K(u)$ inside $K(X_Y)$ must be $K(E), K(u,\sqrt{m})$ and $K(E_m)$, where $E_m$ denotes the quadratic twist of $E$ by $m$. This is impossible, as we know two of these quadratic extensions, namely $K(E)$ and $K(W)$, and we also know that $K(W)$ is not isomorphic to $E$ or $E_m$ over $\overline{K}$. From this, we deduce that $X_Y$ is geometrically connected.

The existence of the Galois diagram asserted by the claim above is immediate by construction.  This completes the proof.   \qedhere

\end{proof}

In the remainder of this subsection, we find the $5$-tuple corresponding to the monodromy of $X_Y \to \mathbb{P}^{1}_{u}$ at its branch points. We use this to find a description of $V_{\ell}(\textup{Jac}_{X_Y})$ as a $G$-module. We first fix some notation.

\begin{notation} \label{notation_for_SpxC2}  We let $\pi_1:S_p \times C_2 \to S_p$ and $\pi_2 : S_p \times C_2 \to C_2$ be the natural projections to the first and second factors. For the remainder of this section, we write $\tau_k$ to be \textit{any} element (well-defined up to conjugacy) in $S_p \times C_2$ satisfying
\begin{enumerate}
\item  $\pi_1(\tau_k)$ has cycle type $(2)^k$ as an element in $S_p$,
\item and $\pi_2(\tau_k)$ is trivial.
\end{enumerate}
Similarly, we define $\tilde{\tau}_{k}$ analogously, but instead we impose that $\pi_2(\tilde{\tau}_{k})$ is surjective instead.
\end{notation}

\begin{proposition} \label{Tate_module_decomposable} 
Let $X_{Y} \to \mathbb{P}^{1}$ be the Galois cover of curves afforded by Proposition \ref{Galois_closure}. Then, $X_{Y} \to \mathbb{P}^{1}_{u}$ is branched at $5$ points, and in addition the monodromy tuple of $X_{Y} \to \mathbb{P}^{1}_{u}$ is $S_p \times C_2$-conjugate to  $(\tilde{\tau}_{\frac{p-1}{2}}, \tilde{\tau}_{\frac{p-1}{2}} , \tilde{\tau}_{\frac{p-1}{2}}, \tilde{\tau}_{\frac{p-3}{2}}, \tau_1)$.

\end{proposition}
\begin{proof}   
In view of \cite{Kuhn}, we deduce that the branch locus of $Y \xrightarrow{\phi} E \xrightarrow{\pi_E} \mathbb{P}_{u}^{1}$ consists of $5$ points. By Proposition \ref{Galois_closure}, $X_Y \to \mathbb{P}^{1}_{u}$ is the Galois closure of  $Y \xrightarrow{\phi} E \xrightarrow{\pi_E} \mathbb{P}_{u}^{1}$, and by using \cite[Corollary 3.9.3(b)]{algebraic_function_fields}, we deduce that the branch locus of $X_Y \to \mathbb{P}^{1}_{u}$ coincides with the branch locus of $Y \to \mathbb{P}^{1}_{u}$. This proves the first half of the claim.

By Remark \ref{remark:decomposition_groups}, the elements in the monodromy tuple coincide with the decomposition groups of any point $y \in X_Y$ lying above a branch point, which in turn coincides with the decomposition group of the corresponding prime ideal $\mathfrak{p}_{y}$ of $K(X_Y)$  over $K(u)$. Given any branch point $d_{i}$, we write $y_{i}$ to denote any point the fibre of $d_{i}$ under $X_{Y} \to \mathbb{P}^{1}_{u}$, and we set $\mathfrak{p}_{i}= \mathfrak{p}_{y_i} \cap K(X_{\mathbb{P}^{1}})$ and $\mathfrak{q}_{i} = \mathfrak{p}_{y_i} \cap K(E)$. In view of Proposition \ref{Galois_closure}, we deduce that the natural restriction maps induce an isomorphism
\begin{equation*} \label{eq:decomposition_groups_product}
\begin{split}
\textup{Gal}(K(X_Y)/K(u)) &\xrightarrow{\sim} \textup{Gal}(K(X_{\mathbb{P}^{1}})/K(u)) \times \textup{Gal}(K(E)/K(u))= S_p \times C_2  \\
\sigma &\mapsto (\sigma \mid_{K(X_{\mathbb{P}^{1}})}, \sigma|_{K(E)}).
\end{split}
\end{equation*}
In addition, the decomposition group $G_{\mathfrak{p}_{y_i}}$ gets identified with $G_{\mathfrak{p}_{i}} \times G_{\mathfrak{q}_{i}}$ under this isomorphism.  By \cite[\S 1]{Kuhn}, the ramification structure of $\Phi: \mathbb{P}^{1}_{x} \to \mathbb{P}^{1}_{u}$ (see Definition \ref{Def:ramification_structure}) in the generic case is $\{ (2)^{\frac{p-1}{2}},(2)^{\frac{p-1}{2}},(2)^{\frac{p-1}{2}},(2)^{\frac{p-3}{2}}, (2)^{1}\}$, and therefore the cycle type of $\{G_{\mathfrak{p}_{i}}\}_{i=1}^{5}$ matches this in view of Theorem \ref{theorem_miranda}. In addition, the first four of these points correspond to the four branch points of $E \to \mathbb{P}^{1}_{u}$, and therefore $G_{\mathfrak{q}_{i}}$ is all of $C_2$ for $i=1,2,3,4$, while it's trivial for $i=5$. This gives the required result. \qedhere

\end{proof}
Following the notation introduced in Theorem \ref{equivariant_riemann_hurwitz}, given $\sigma \in S_p \times C_2$, we write $\mathbf{1}_{\sigma}^{*}$ to denote the permutation representation $\textup{Ind}_{\langle \sigma \rangle} ^{S_p \times C_2} \mathbf{1}$.
\begin{theorem} \label{Tate_module_X_Y}Let $X_{Y} \to \mathbb{P}^{1}$ be the Galois cover of curves afforded by Proposition \ref{Galois_closure}. For any prime number $\ell$, the following $G$-representations
\[ V_{\ell}(\textup{Jac}_{X_{Y}}) \ \ \textup{and} \ \ \mathbf{1}^{\oplus
2} \oplus (\mathbf{1}_{e}^{*})^{\oplus 3} \ominus \mathbf{1}_{\tau_1}^{*}  \ominus
\mathbf{1}_{\tilde{\tau}_{\frac{p-3}{2}}}^{*} \ominus(\mathbf{1}_{\tilde{\tau}_{\frac{p-1}{2}}}^{*})^{\oplus 3} \]
become isomorphic after extending scalars to $\mathbb{C}$.
\end{theorem}
\begin{proof}
The claim follows by combining Proposition \ref{Tate_module_decomposable} with Corollary \ref{eq_riemann_hurwitz_gc} (which is applicable since $X_{Y}/K$ is geometrically connected  in view of Proposition \ref{Galois_closure}). \qedhere
\end{proof}

\subsection{Complimentary elliptic curve $E'$ and a local expression}
The main result in this subsection is Theorem \ref{complementary_elliptic_curve} below which gives a canonical choice for the complimentary elliptic curve $E' $ as the Jacobian of the quotient of $X_Y$ by a subgroup $H' \leq S_p \times C_2$. We then prove pseudo Brauer verifiability for the underlying isogeny $\textup{Jac}_{Y} \to E \times E'$, and we end this section by obtaining a local expression for the parity of the $p^{\infty}$-Selmer rank of $\textup{Jac}_{Y}$ using regulator constants.
\begin{notation} \label{def_of_H'}We view $S_{p} \times C_{2}$ as a subgroup of $S_{p+2}$ by letting it act on $\{1,\dots,p,p+1,p+2\}$ in the natural way, i.e. the $S_p$ factor acts on $\{1,\dots,p\}$ in the natural way, while $C_2$ permutes $p+1 \leftrightarrow p+2$. For the remainder of this section, we fix $ H '$ to be the subgroup of $S_{p} \times C_{2}$ generated by the permutations
$(1,2,\dots,p-2), (1,2), (p-1,p)(p+1,p+2)$ if $p \geq 5$, while we write $H'$ to denote the subgroup generated by $(p-1,p)(p+1,p+2)$ if $p \leq 3$. As an abstract group, $H'$ is isomorphic to $S_{p-2} \times C_2$. We note that in the case $p=2$, this coincides with the subgroup generated by $rh_{Y}$ from Proposition \ref{elliptic_involution}. \end{notation}

For the remainder of this section, we write $\epsilon$ to denote the $S_p\times C_2$-representation obtained by lifting the non-trivial representation from the $C_{2}$ quotient, while $\sigma$ denotes the representation obtained by lifting the standard representation $(p-1)$-dimensional representation from the $S_{p}$ quotient.

\begin{theorem} \label{complementary_elliptic_curve}
Let $X_{Y} \to \mathbb{P}^{1}$ be the Galois cover of curves afforded by Proposition \ref{Galois_closure}. Then, $\Theta=( S_{p-1}\times\{e\})-( S_{p}\times\{e\})-H'$ is a pseudo Brauer relation for $S_p \times C_2$ and $X_Y$ which verifies an isogeny $\textup{Jac}_{Y} \to E \times \textup{Jac}_{X_Y/H'}.$
\end{theorem}
\begin{proof} 
The case where $p=2$ follows from Proposition \ref{elliptic_involution}.
For brevity, write $V_{\ell}$ to denote the $G$-representation $V_{\ell}(\textup{Jac}_{X_Y}) \otimes \mathbb{C}$, and we write $A_1,A_2$ and $A_3$  to denote the subgroups $S_{p}\times \{e\}, S_{p-1}\times C_2$ and $S_{p-1} \times \{e\}$ of $S_p \times C_2$ (as in Proposition \ref{Galois_closure}). In addition, we write $\mathbf{1}_{H}^{*}$ to denote the representation $\textup{Ind}_{H}^{S_p \times C_2} \mathbf{1}$. Then, decomposing into irreducible representations, we get
\[\mathbf{1}_{A_1}^{*} = \mathbf{1} \oplus \epsilon, \ \ \mathbf{1}_{A_2}^{*} = \mathbf{1} \oplus \sigma, \ \ \mathbf{1}_{A_3}^{*} = \mathbf{1} \oplus \epsilon \oplus \sigma \oplus (\sigma\otimes \epsilon), \]
By \cite[Lemma 3.18]{DGKM}, we have $\langle V_{\ell}, \mathbf{1}_{H}^{*} \rangle = 2 \cdot\textup{genus}(X_Z/H).$ The genera of the quotients of $X_Z$ by $A_1,A_2$ and $A_3$ are known in view of Proposition \ref{Galois_closure} allowing one to deduce that $
  \langle V_{\ell}, \mathbf{1} \rangle = \langle V_{\ell}, \sigma \rangle =0,$ while $\langle V_{\ell}, \epsilon \rangle =\langle V_{\ell}, \sigma \otimes \epsilon \rangle=2.$
In addition, for $p\geq 3$, 
\begin{equation*}
\#\langle \tau \rangle \backslash (S_p \times C_2) / H' = \begin{cases}
p^2-3p+3, & \tau = \tau_1, \\
\frac{1}{2}(p^2-3), & \tau=\tilde{\tau}_{\frac{p-3}{2}}, \\
\frac{1}{2}(p^2-1), & \tau=\tilde{\tau}_{\frac{p-1}{2}}.
\end{cases}
\end{equation*} 
Therefore, by invoking Theorem \ref{Tate_module_X_Y}, we deduce that $\langle V_{\ell}, \mathbf{1}_{H'}^{*} \rangle =2$, and therefore $\textup{genus}(X_Y/H') = 1.$ Moreover,
\[ \langle \sigma\otimes \epsilon, \mathbf{1}_{H'}^{*}  \rangle = \langle \mathbf{1}_{A_3}^{*}  - \mathbf{1}_{A_2}^{*} - \mathbf{1}_{A_1}^{*} + \mathbf{1}, \mathbf{1}_{H'}^{*}  \rangle = 3-2-1+1=1.  \]
Finally, by writing $\mathbf{1}_{H'}^{*} = (\sigma \otimes \epsilon) \oplus T$, we deduce that $\langle V_{\ell} ,T \rangle =0$ as we have shown that $\langle V_{\ell}, \sigma \otimes \epsilon \rangle =\langle V_{\ell}, \mathbf{1}_{H'}^{*} \rangle =2$. From this we deduce that $\Theta= ( S_{p-1}\times\{e\})-( S_{p}\times\{e\})-H'$ is a pseudo Brauer relation for $S_p \times C_2$ and $X_{Y}$. In view of the quotient curves in Proposition \ref{Galois_closure}, $\Theta$ verifies an isogeny $\textup{Jac}_{Y} \to E \times \textup{Jac}_{X_Y/H'}.$ \qedhere

\end{proof}

We now use the pseudo Brauer relation from Theorem \ref{complementary_elliptic_curve} in order to provide a local expression for $\textup{rk}_{p} \ \textup{Jac}_{Y}$. We first prove some auxiliary lemmata.

\begin{lemma} \label{multiplicities_in_selmer}Let $p$ be a rational prime. Then, the multiplicities of $\epsilon$ and $\sigma \otimes \epsilon$ in $\mathcal{X}_{p}(\textup{Jac}_{X_Y})$ satisfy
\begin{equation*}
\langle \mathcal{X}_{p}(\textup{Jac}_{X_Y}), \epsilon \rangle = \textup{rk}_{p}(E) \ \ \ \textup{and} \ \ \  \langle \mathcal{X}_{p}(\textup{Jac}_{X_Y}), \sigma \otimes \epsilon \rangle = \textup{rk}_{p}(\textup{Jac}_{X_Y/H'}),
\end{equation*}
\begin{proof}
We start by writing $\epsilon = \mathbf{1}_{A_1}^{*}  - \mathbf{1}$. Then, $\langle \mathcal{X}_{p} , \epsilon \rangle$ satisfies
\[
\langle \mathcal{X}_{p}, \mathbf{1}_{A_1}^{*}  - \mathbf{1} \rangle 
\underset{\mathrm{Frob. rec.}}{=}\textup{dim} \ \mathcal{X}_{p}^{A_1}  - \textup{dim} \ \mathcal{X}_{p}^{S_p \times C_2}  \underset{\textup{Prop.} \ref{Galois_closure}}{=}\textup{dim} \ \mathcal{X}_{p}(E).
\]
The second claim follows in the same way after expressing  $\sigma\otimes \epsilon = \mathbf{1}_{A_3}^{*}  - \mathbf{1}_{A_2}^{*} - \mathbf{1}_{A_1}^{*} + \mathbf{1}$, and therefore $\langle \mathcal{X}_{p}, \sigma \otimes \epsilon \rangle = \textup{rk}_{p}(\textup{Jac}_{Y}) - 0 - \textup{rk}_{p}(E) +0 = \textup{rk}_{p}(\textup{Jac}_{X_Y/H'})$ as required. \qedhere
\end{proof}

\end{lemma}
\begin{lemma} \label{regulator_constants_non-morally-zero-reps} Let $\Theta=( S_{p-1}\times\{e\})-( S_{p}\times\{e\})-H'$ be the pseudo Brauer relation for $X_Y/K$ afforded by Theorem \ref{complementary_elliptic_curve}. Then,
$\mathcal{C}_{\Theta}(\epsilon) \equiv \mathcal{C}_{\Theta}(\sigma \otimes \epsilon) \equiv p \mod \mathbb{Q}^{\times 2}. $
\end{lemma}
\begin{proof} The calculation for $\mathcal{C}_{\Theta}(\epsilon)$ is straightforward (one can choose the trivial pairing on $\epsilon$). We fix a basis $\{e_1,\dots,e_p\}$ for $\mathbb{C}^{p}$. Then, $\sigma \otimes \epsilon$ corresponds to the subspace spanned by $v_{i}=e_{i+1}-e_1$ for $i=1,\dots,p-1$ with $(\rho_1,\rho_2) \in S_{p} \times C_{2}$ acting via $(\rho_1,\rho_2)\cdot v_{i}= \textup{sgn}(\rho_2)(e_{\rho_1(i+1)}-e_{\rho_1(1)})$. A non-degenerate, $G$-invariant pairing on $\sigma \otimes \epsilon$ is given by setting the diagonal entries
$ \Langle  v_{i}, v_{i} \Rangle  =2$ while $\Langle v_{i}, v_{j} \Rangle =1$ for $i \neq j$. Suppose that the $S_{p-1}$ factor in $S_{p-1} \times \{e\}$ stabilises $1$. Then for $H\in \{S_{p-1} \times \{e\}, S_{p} \times \{e\}, H'\}$, the space of $H$-invariant vectors is $1$-dimensional spanned by $u_1=\sum_{i=1}^{p-1} v_i$, $0$-dimensional and $1$-dimensional spanned by $u_2=v_{p-2} - v_{p-1}$ respectively. Therefore, $\Langle u_1, u_1 \Rangle = p(p-1)$, while $\Langle u_2, u_2 \Rangle =2$. Putting everything together,
\[\mathcal{C}_{\Theta}(\sigma\otimes \epsilon) \equiv \frac{\frac{1}{(p-1)!}\Langle u_1,u_1 \Rangle}{\frac{1}{2(p-2)!}\Langle u_2,u_2 \Rangle} \equiv p \mod \mathbb{Q}^{\times 2}. \qedhere \]

\end{proof}
\begin{theorem} \label{Thm:local_formula_genus2_split} Let $Y/K$ be a genus $2$ curve which admits a $K$-rational morphism $\phi:Y\to E$ of prime degree $p$. In addition, if $p$ is an odd prime, suppose that $\phi$ is generic. Then,
\[\textup{rk}_{p}(\textup{Jac}_{Y}) \equiv \sum_{v \ \textup{place of} \ K} \textup{ord}_{p} \ \Lambda_{\Theta}(X_{Y}/K_{v}) \ \mod 2, \]where $\Theta=(S_{p-1}\times \{e\})-(S_p \times \{e\})-H'$ is the pseudo Brauer relation for $S_p \times C_2$ and $X_{Y}$ afforded by Theorem \ref{complementary_elliptic_curve}.
\end{theorem}
\begin{proof}
The $p=2$ case is dealt with in Proposition \ref{elliptic_involution}. We write $\mathcal{X}_p$ to denote $\mathcal{X}_{p}(\textup{Jac}_{X_Y})$. Then, we claim that $S_{\Theta,p} = \{\epsilon, \sigma\otimes \epsilon \}.$ Indeed, by combining the proof of Theorem \ref{complementary_elliptic_curve} with Lemma \ref{trivial_regulator_constant}(1), we deduce that $S_{\Theta,p} \subseteq \{\epsilon, \sigma\otimes \epsilon\}$. By Lemma \ref{regulator_constants_non-morally-zero-reps}, it follows that $S_{\Theta,p} = \{\epsilon, \sigma\otimes \epsilon\}$.   Since any $S_p\times C_2$-representation is self-dual, then Theorem \ref{regulator_constant_rank_parity} is applicable. In view of Lemma \ref{multiplicities_in_selmer}, the result follows upon applying Theorem \ref{regulator_constant_rank_parity} to $\Theta$ afforded by Theorem \ref{complementary_elliptic_curve}.\end{proof}
\begin{remark}
We note that elliptic subcovers occur in pairs $(\phi,\psi)$. In particular, given a genus $2$ curve admitting an elliptic subcover $\phi:Y \to E$, then there exists a complimentary elliptic subcover $\psi:Y \to E'$ of the same degree, see \cite{Kuhn}. Therefore, upon interchanging the roles of $\phi$ and $\psi$ and $E$ with $E'$ if needed, we note that the conditions for the validity of Proposition \ref{Galois_closure}, Theorem \ref{complementary_elliptic_curve} and Theorem \ref{Thm:local_formula_genus2_split} may be relaxed to apply when $\phi$ or its complementary cover $\psi$ conforms to the condition of being generic.
\end{remark}

\begin{remark}
Drawing from customary techniques, we can derive an analogous local formula in this case is as follows. We let $\phi: Y \to E$ be a degree $p$ elliptic subcover, and write $\psi:Y \to E'$ for the complimentary elliptic subcover. Then, $F=\phi^{*}+\psi^{*}$ is a $(p,p)$ isogeny satisfying $F \circ F^{\vee} = [p]_{\textup{Jac}_{Y}}$. By invoking \cite[Theorem 5]{Yu}, we deduce
\[\textup{rk}_{p}(\textup{Jac}_{Y}) \equiv \sum_{v }  \textup{ord}_{p} \ \frac{C_{v} \ (\textup{Jac}_{Y}, \phi^{*} \omega_{v}^{0}(E) \wedge \psi^{*} \omega_{v}^{0}(E')   )}{c_{v} (E)c_{v}({E'})} \cdot \theta_{v} \mod 2, \] where  $\omega_{v}^{0}(E)$ and $\omega_{v}^{0}(E')$ denote a choice of minimal differentials on $E/K_{v}$ and ${E'}/K_{v}$ respectively, while $\theta_{v} =  \frac{\mu_{v}(Y)}{\mu_{v}(E) \cdot \mu_{v}(E')}$ if $p=2$, and otherwise $\theta_v=1$. We note that this method for obtaining this local expression for $\textup{rk}_{p}(\textup{Jac}_{Y})$ makes no assumption on $\phi$  or $\psi$ being generic. \end{remark}

\bibliographystyle{plain}

\bibliography{references}
\end{document}